\documentclass{amsart}
\usepackage[body={6in, 8in}, a4paper]{geometry}
\usepackage[all]{xy}
\usepackage{amssymb, amsmath, color, hyperref}
\hypersetup{colorlinks=true}

\makeatletter
\def\url@leostyle{%
  \@ifundefined{selectfont}{\def\UrlFont{\sf}}{\def\UrlFont{\small\ttfamily}}}
\makeatother
\definecolor{dblue}{rgb}{0,0,0.7}
\newtheoremstyle{mythm}{11pt}{11pt}{\it\color{dblue}}{}{\bf\color{dblue}}{.}{ }{}
\theoremstyle{mythm}
\newtheorem{theorem}{Theorem}[section]
\newtheorem{lemma}[theorem]{Lemma}
\newtheorem{corollary}[theorem]{Corollary}
\newtheorem{proposition}[theorem]{Proposition}

\theoremstyle{definition}

\theoremstyle{remark}
\newtheorem{remark}[theorem]{Remark}

\font\russ=wncyr10  1
\def\sha{\hbox{\russ\char88}}

\newcommand{\eins}{\boldsymbol{1}}
\DeclareMathOperator{\Aut}{Aut}

\DeclareMathOperator{\Ind}{Ind}
\DeclareMathOperator{\Gal}{Gal}
\DeclareMathOperator{\Hom}{Hom}

\DeclareMathOperator{\Spec}{Spec}

\DeclareMathOperator{\res}{res}
\DeclareMathOperator{\im}{im}
\DeclareMathOperator{\Sel}{Sel}
\DeclareMathOperator{\diag}{diag}
\DeclareMathOperator{\Fit}{Fit}
\DeclareMathOperator{\Fr}{Fr}
\DeclareMathOperator{\GL}{GL}
\DeclareMathOperator{\Indshort}{I}
\DeclareMathOperator{\Tr}{Tr}

\DeclareMathOperator{\Nrd}{Nrd}
\DeclareMathOperator{\Reg}{Reg}

\DeclareMathOperator{\cok}{cok}

\DeclareMathOperator{\Ir}{Ir}
\DeclareMathOperator{\rk}{rk}

\newcommand{\CC}{\mathbb{C}}
\newcommand{\FF}{\mathbb{F}}

\newcommand{\QQ}{\mathbb{Q}}
\newcommand{\RR}{\mathbb{R}}
\newcommand{\ZZ}{\mathbb{Z}}

\newcommand{\id}{\mathrm{id}}

\newcommand{\Lstar}{L^{\raisebox{2pt}{$\scriptstyle\star$}}}               
\newcommand{\calLstar}{\mathcal{L}^{\raisebox{1pt}{$\scriptstyle\star$}}}

\newcounter{condone}

\newenvironment{conditions}{\begin{list}{(\alph{condone})}{\usecounter{condone}}}{\end{list}}

\begin{document}

\title[Mordell-Weil and congruences]{On Mordell-Weil groups\\ and congruences between derivatives \\of twisted Hasse-Weil $L$-functions}

\author{David Burns, Daniel Macias Castillo and Christian Wuthrich}

\address{King's College London, Department of Mathematics, London WC2R 2LS,
U.K.}
\email{david.burns@kcl.ac.uk}

\address{Instituto de Ciencias Matem\'aticas (ICMAT), 28049 Madrid, Spain.}
\email{daniel.macias@icmat.es}

\address{School of Math. Sciences,
University of Nottingham,
Nottingham NG7 2RD,
U.K.}
\email{christian.wuthrich@nottingham.ac.uk}

\begin{abstract} Let $A$ be an abelian variety defined over a number field $k$ and let $F$ be a finite Galois extension of $k$. Let $p$ be a prime number. Then under certain not-too-stringent conditions on $A$ and $F$ we compute explicitly the algebraic part of the $p$-component of the equivariant Tamagawa number of the pair $\bigl(h^1(A_{/F})(1),\ZZ[{\rm Gal}(F/k)]\bigr)$. By comparing the result of this computation with the theorem of Gross and Zagier we are able to give the first verification of the $p$-component of the equivariant Tamagawa number conjecture for an abelian variety in the technically most demanding case in which the relevant Mordell-Weil group has strictly positive rank and the relevant field extension is both non-abelian and of degree divisible by $p$. More generally, our approach leads us to the formulation of certain precise families of conjectural $p$-adic congruences between the values at $s=1$ of derivatives of the Hasse-Weil $L$-functions associated to
twists of $A$, normalised by a product of explicit equivariant regulators and periods, and to explicit predictions concerning the Galois structure of Tate-Shafarevich groups. In several interesting cases we provide theoretical and numerical evidence in support of these more general predictions.
\end{abstract}

\maketitle

\section{Introduction}\label{Intro}

Let $A$ be an abelian variety defined over a number field $k$. Then for any finite Galois extension $F$ of $k$ with group $G$ the equivariant Tamagawa number conjecture for the pair $\bigl(h^1(A_{/F})(1),\ZZ[G]\bigr)$ is formulated in \cite[Conjecture 4(iv)]{bufl99} as an equality in a relative algebraic $K$-group. This conjectural equality is rather technical to state and very inexplicit in nature but is known to constitute a strong and simultaneous refinement of the Birch and Swinnerton-Dyer conjecture for $A$ over each of the intermediate fields of $F/k$.

The refined conjecture also naturally decomposes into `components', one for each rational prime $p$, in a way that will be made precise in \S\ref{generalcase} below, and each such $p$-component (which for convenience we refer to as `eTNC$_p$' in the remainder of this introduction) is itself of some interest. We recall, for example, that if $A$ has good ordinary reduction at $p$, then the compatibility result proved by Venjakob and the first named author in~\cite[Theorem 8.4]{burns_venjakob} shows that eTNC$_p$ is very closely related to the main conjecture of non-commutative Iwasawa theory for $A$ with respect to any compact $p$-adic Lie extension of $k$ that contains $F$.

If $p$ does not divide the order of $G$, then, without any hypothesis on the reduction of $A$ at $p$, it is straightforward to use the techniques developed in~\cite[\S1.7]{bufl95} to give an explicit interpretation of eTNC$_p$ (although, of course, obtaining a full proof in this case still remains a very difficult problem). However, if $p$ divides the order of $G$, then even obtaining an explicit interpretation of eTNC$_p$ has hitherto seemed to be a very difficult problem - see, for example, the considerable efforts made by Bley~\cite{Bley, Bley2} in this direction.

One of the main goals of the present article is therefore to develop a general approach which can be used to obtain (both theoretical and numerical) verifications of eTNC$_p$ in the technically most demanding case in which the abelian variety $A$ has strictly positive rank over $F$ and the Galois group $G$ is both non-abelian and of order divisible by $p$.

To give a concrete example of the results we shall obtain in this way, we note that in \S\ref{proofmain} our techniques will be combined with the theorem of Gross and Zagier to prove the following result (and to help put this result into context see Remarks \ref{satisfied} and \ref{loub} below).

\begin{theorem}\label{HPintro}
Let $p$ be an odd prime, $K$ an imaginary quadratic field in which $p$ is unramified and $F$ its Hilbert $p$-classfield. Write $G$ for the generalised dihedral group $\Gal(F/\QQ)$.

Let $A$ be an elliptic curve over $\QQ$ which satisfies, with respect to the field $K$, the hypotheses~\ref{hyp_a}, \ref{hyp_b}, \ref{hyp_c}, \ref{hyp_e} and \ref{hyp_f} that are listed in \S\ref{sect hypotheses} and for which the Tate-Shafarevich group $\sha(A_F)$ is finite.

Assume in addition that each of the following four hypotheses is satisfied.
\begin{enumerate}
   \item[(i)] \label{HPintro_i} all primes of bad reduction for $A$ split in $K/\QQ$;
   \item[(ii)] \label{HPintro_ii} the Galois group of $K\bigl(A[p]\bigr)/K$ is isomorphic to $\GL_2(\FF_p)$;
   \item[(iii)] \label{HPintro_iii} the Hasse-Weil $L$-function $L(A_{/K},s)$ has a simple zero at $s=1$;
   \item[(iv)] \label{HPintro_iv} the trace to $A(K)$ of the Heegner point is not divisible by $p$.
   \end{enumerate}

Then the $p$-component of the equivariant Tamagawa number conjecture is valid for the pair $\bigl(h^1(A_{/F})(1),\ZZ[G]\bigr)$.
\end{theorem}
\begin{remark}\label{satisfied}
For a fixed elliptic curve $A$ over $\QQ$ and imaginary quadratic field $K$, it will be clear that the hypotheses~\ref{hyp_a}, \ref{hyp_b} and \ref{hyp_c} that occur in the statement of Theorem \ref{HPintro} are satisfied by all but finitely many primes $p$. In addition, it will be clear that for a fixed $A$ and $p$ the hypotheses~\ref{hyp_e} and \ref{hyp_f} only constitute a mild restriction on the ramification of the extension $K/\QQ$.
Finally we recall that, by a famous result of Serre~\cite{serre}, the hypothesis (ii) holds for almost all $p$ if $A$ does not have complex multiplication.
\end{remark}

\begin{remark}\label{loub} In the context of Theorem~\ref{HPintro} we also note that for any natural number $n$ there are infinitely many imaginary quadratic fields $K$
in which $p$ does not
ramify and for which the group $\Gal(F/K)$ has exponent divisible by $p^n$. (For details of an explicit, and infinite, family of such fields see for example Louboutin~\cite{sl}.) In this context it therefore seems worth noting that the only previous verification of eTNC$_p$ for an abelian variety $A$ and a Galois extension of degree divisible by $p$ is given by Bley in~\cite[Corollary 1.4 and Remark 4.1]{Bley3} where it is assumed, amongst other things, that $A$ is an elliptic curve, $F$ is an abelian extension of $\QQ$ of exponent $p$ and, critically, that the rank of $A$ over $F$ is equal to zero. \end{remark}

In addition to the proof of Theorem \ref{HPintro}, in \S \ref{special cases} we will also provide further explicit examples in which our approach leads to a proof of eTNC$_p$ in cases for which $p$ divides the degree of the relevant extension. More precisely, we show that for certain elliptic curves $A$ the validity of eTNC$_p$ with $p = 3$ follows from that of the relevant cases of the Birch and Swinnerton-Dyer
conjecture for a family of $S_3$-extensions of number fields (see Corollary \ref{arith_cor}) and provide examples (with $p =5$ and $p=7$) in which our approach allows eTNC$_p$ to be verified by numerical computations (see \S \ref{numex}).

In order to prove the above results we must first establish several intermediate results which are both more general and also we feel of some independent interest. To briefly discuss these results, we fix an abelian variety $A$ over a number field $k$, a finite Galois extension $F$ of $k$
with Galois group $G$ and an odd prime $p$.

In \S \ref{bloch-kato} we shall first compute, under some not-too-stringent conditions on $A$ and $F$, the `algebraic part' of the $p$-component of the equivariant Tamagawa number of the pair $\bigl(h^1(A_{/F})(1),\ZZ[G]\bigr)$. This computation requires a close analysis of certain refined Euler characteristics (in the sense reviewed in \S\ref{ec}) that are constructed by combining the finite support cohomology complex of Bloch and Kato for the base change through $F/k$ of the $p$-adic Tate module of the dual $A^t$ of $A$ together with the N\'eron-Tate height of $A$ relative to the field $F$. We note that the main result of this analysis constitutes a natural equivariant refinement and/or generalisation of several earlier computations in this area including those that are made by Venjakob in~\cite[\S3.1]{venjakob}, by Bley in~\cite{Bley}, by Kings in~\cite[Lecture 3]{kings} and by the first named author
in~\cite[\S12]{ltav}.

The detailed computation made in \S\ref{bloch-kato} then allows us in \S\ref{congruences section} to interpret the relevant case of eTNC$_p$ as a
family of $p$-adic congruence relations between the values at $s=1$ of higher derivatives of the Hasse-Weil $L$-functions of twists of $A$ by irreducible complex characters of $G$, suitably normalised by a product of explicit equivariant regulators and periods.

We next describe several interesting (conjectural) consequences of this reinterpretation of eTNC$_p$, including the predicted annihilation as a Galois
module of the $p$-primary Tate-Shafarevich of $A^t$ over $F$ by elements which interpolate the (suitably normalised) values at $s=1$ of higher derivatives of twisted Hasse-Weil $L$-functions (see Proposition \ref{consequences} and Proposition \ref{etnc proj case}).

By using results on the explicit Galois structure of Selmer groups that are obtained in \cite{Selmerstr}, we then show that for generalised dihedral groups $G$ the necessary $p$-adic congruence relations can
be made very explicit even in the case that $A$ has strictly positive rank over $F$ (see Theorem \ref{exp cong}).

The latter explicit interpretation is then used as a key step in the proof of Theorem \ref{HPintro}.

We note finally that there are several ways in which it seems reasonable to expect that some of the explicit computations made in \S\ref{bloch-kato} could (with perhaps considerably more effort) be extended and our overall approach thereby generalised and that we hope such possibilities will be further explored in subsequent articles.

\subsection{Acknowledgements}
The authors are very grateful to Werner Bley and Stefano Vigni for helpful discussions and correspondence and to the anonymous referee for making several
useful suggestions.

\section{Notations and setting}
For any finite group $\Gamma$ we use the following notation. We write $\Ir(\Gamma)$ for the set of irreducible $E$-valued characters  of $\Gamma$, where $E$ denotes
either $\CC$ or $\CC_p$ (and the intended meaning will always be clear from the context). We also write $\eins_\Gamma$ for the trivial character of $\Gamma$ and
$\check\psi$ for the contragredient of each $\psi$ in $\Ir(\Gamma)$. For each $\psi$ in $\Ir(\Gamma)$ we write
$$
 e_\psi = \frac{\psi(1)}{|\Gamma|}\, \sum_{\gamma \in \Gamma}\psi(\gamma^{-1})\,\gamma
$$
for the primitive idempotent of the centre $\zeta\bigl(E[\Gamma]\bigr)$ of the group ring $E[\Gamma]$. For each $E$-valued character $\psi$ we also fix an $E[\Gamma]$-module
$V_\psi$ of character $\psi$.

For any abelian group $M$ we write $M_{\rm tor}$ for its torsion subgroup and $M_{\rm tf}$ for the quotient $M/M_{\rm tor}$, which we often regard as a subgroup of
$\QQ\otimes_\ZZ M$. For any prime $p$ we write $M[p]$ for the subgroup $\{m \in M: pm =0\}$ of the Sylow $p$-subgroup $M[p^{\infty}]$ of $M_{\rm tor}$.
We also set $M_p := \ZZ_p\otimes_\ZZ M$ and write $M^{\wedge}_p$ for the  pro-$p$-completion $\varprojlim_n M/p^n M$. If $M$ is finitely generated, then we set
$\rk(M):= \dim_{\QQ}(\QQ\otimes_\ZZ M)$.

For any $\ZZ_p[\Gamma]$-module $M$ we write $M^\vee$ for the Pontryagin dual $\Hom_{\ZZ_p}(M,\QQ_p/\ZZ_p)$ and $M^*$ for the linear dual $\Hom_{\ZZ_p}(M,\ZZ_p)$, each
endowed with the natural contragredient action of $\Gamma$. If $M$ is finitely generated, then for any field extension $F$ of $\QQ_p$ we set
$F\cdot M := F\otimes_{\ZZ_p}M$.

For any Galois extension of fields $L/K$ we write $G_{L/K}$ in place of $\Gal(L/K)$. We also fix an algebraic closure $K^c$ of $K$ and abbreviate $G_{K^c/K}$ to $G_K$.
For each non-archimedean place $v$ of a number field we write $\kappa_v$ for its residue field.

Throughout this paper, we will consider the following situation. We have a fixed odd prime $p$ and a Galois extension $F/k$ of number fields with Galois group $G=G_{F/k}$.
We choose a $p$-Sylow subgroup $P$ in $G$ and set $K := F^P$. We give ourselves an abelian variety $A$ of dimension $d$ defined over $k$. We write $A^t$ for the dual
abelian variety of $A$.

For each intermediate field $L$ of $F/k$ we write $S_p^L$, $S^L_{\rm r}$ and $S^L_{\rm b}$ for the set of non-archimedean places of $L$ that are $p$-adic, which ramify
in $F/L$ and at which $A_{/L}$ has bad reduction respectively. Similarly, we write $S_\infty^L$, $S^L_\RR$ and $S^L_\CC$ for the set of archimedean, real and complex
places of $L$ respectively.
The notation $S_{\text{?}}$ will stand for $S_{\text{?}}^k$ when ``?'' is b, r, $p$, $\infty$, $\RR$ or $\CC$.

For each such field $L$ we also set $\rk(A_L) := \rk(A(L))$ and write $\sha_p(A_L)$ and $\Sel_p(A_{L})$ for the $p$-primary
Tate-Shafarevich and Selmer groups of $A_{/L}$. We recall that there exists a canonical exact sequence of the form
\begin{equation}\label{sha-selmer}
 \xymatrix{0 \ar[r] & \sha_p(A_L)^\vee \ar[r] & \Sel_p(A_L)^\vee \ar[r]&  A(L)_p^*\ar[r] & 0.}
\end{equation}

\section{The hypotheses}\label{sect hypotheses}

Fix an odd prime $p$, number fields $k, F$ and $K = F^P$ and an abelian variety $A$ as described above.

Then throughout this article we will find it convenient to assume that this data satisfies the following hypotheses:

\begin{conditions}
\item\label{hyp_a}   $A(K)[p] = 0$ and $A^t(K)[p] = 0$;
\item\label{hyp_b}  The Tamagawa number of $A_{/K}$ at each place in $S_{\rm b}^K$ is not divisible by $p$;
\item\label{hyp_c} $A_{/K}$ has good reduction at all $p$-adic places;
\item\label{hyp_d}  For all $p$-adic places $v$ that ramify in $F/K$, the reduction is ordinary and $A(\kappa_v)[p]= 0$;
\item\label{hyp_e}   No place of bad reduction for $A_{/k}$ is ramified in $F/k$, i.e. $S_{\rm b} \cap S_{\rm r} = \emptyset$.
\item\label{hyp_f} For any place $v$ in $K$ such that the primes above it in $F$ are ramified in $F/k$, we have $A(\kappa_v)[p]=0$;
\item\label{hyp_g} The Tate-Shafarevich group $\sha(A_{F})$ is finite.
\end{conditions}

\begin{remark}
 For a fixed abelian variety $A$ over $k$ and extension $K/k$ the hypotheses~\ref{hyp_a}, \ref{hyp_b} and \ref{hyp_c} are clearly satisfied by all but finitely many primes $p$ (which do not divide the degree of $K/k$), the hypotheses~\ref{hyp_e} and \ref{hyp_f} constitute a mild restriction on the ramification of $F/k$ and
 the hypothesis~\ref{hyp_g} is famously conjectured to be true in all cases. However, the hypothesis~\ref{hyp_d}
 excludes the case that is called `anomalous' by Mazur in~\cite{m} and, for a given $A$, there may be infinitely many primes $p$ for which there are $p$-adic places
 $v$ at which $A$ has good ordinary reduction but $A(\kappa_v)[p]$ does not vanish.
Nevertheless, it is straightforward to describe examples of abelian varieties $A$ for which there are only finitely many such anomalous places -- see, for example,
the result of Mazur and Rubin in~\cite[Lemma A.5]{mr}.
\end{remark}

\begin{remark} In the analysis of finite support cohomology complexes that is given in Lemma \ref{canisos} below we find it convenient to adopt the convention of considering the $p$-adic Tate module of $A^t$ rather than that of $A$ itself. Applications of the hypotheses~\ref{hyp_b}--\ref{hyp_f} in this article will therefore often take place with
$A$ replaced by $A^t$. In this regard, we note that the validity of each of these hypotheses as currently formulated is equivalent
to the validity of the corresponding hypothesis with $A$ replaced by $A^t$ and we will henceforth use this fact without further explicit comment.
\end{remark}


\section{Canonical Euler characteristics}\label{bloch-kato}

In most of this section we assume that the fixed data $p, k,F,K$ and $A$ satisfy all of the hypotheses~\ref{hyp_a}--\ref{hyp_g} that are listed in \S\ref{sect hypotheses}.

Our main aim is to compute explicitly the natural (refined) Euler characteristic that is associated to the pair comprising the Bloch-Kato finite support cohomology complex of the base-change through $F/k$ of the $p$-adic Tate module of the dual abelian variety $A^t$ of $A$ and the N\'eron-Tate height of $A$ relative to the field $F$.

\subsection{Euler characteristics}\label{ec} We first quickly review the definition of refined Euler characteristics that will play an essential role in the sequel. For convenience, we shall only give an explicit construction in the relevant special case rather than discussing the general approach (which is given in~\cite{ewt})

For any finite group $\Gamma$ we write $D\bigl(\ZZ_p[\Gamma]\bigr)$ for the derived category of complexes of (left) $\ZZ_p[\Gamma]$-modules. We also write $D^{\rm p}\bigl(\ZZ_p[\Gamma]\bigr)$ for the full triangulated subcategory of $D\bigl(\ZZ_p[\Gamma]\bigr)$ comprising complexes that are `perfect' (that is, isomorphic in $D\bigl(\ZZ_p[\Gamma]\bigr)$ to a bounded complex of finitely generated projective $\ZZ_p[\Gamma]$-modules).

We assume to be given a complex $C$ in $D^{\rm p}\bigl(\ZZ_p[G]\bigr)$ which is acyclic outside degrees $a$ and $a+1$ for any given integer $a$ and such that $H^a(C)$ is $\ZZ_p$-free. We also assume to be given an isomorphism of $\CC_p[G]$-modules $\lambda: \CC_p\otimes_{\ZZ_p}H^a(C) \cong \CC_p\otimes_{\ZZ_p}H^{a+1}(C)$.
Then, under these hypotheses, it can be shown that there exists an isomorphism $\iota:P^\bullet \cong C$ in $D^{\rm p}\bigl(\ZZ_p[G]\bigr)$ where $P^\bullet$ is a complex of $\ZZ_p[G]$-modules of the form $P\xrightarrow{d} P$ where $P$ is finitely generated and projective and the first term occurs in degree $a$. We then consider the following auxiliary composite isomorphism of $\CC_p[G]$-modules
\begin{multline*}
 \lambda_{P^\bullet}: \CC_p\otimes_{\ZZ_p}P \cong \bigl(\CC_p\otimes_{\ZZ_p}H^a(P^\bullet)\bigr) \oplus  \bigl(\CC_p\otimes_{\ZZ_p}\im(d)\bigr)\\
\cong \bigl(\CC_p\otimes_{\ZZ_p}H^{a+1}(P^\bullet)\bigr) \oplus \bigl(\CC_p\otimes_{\ZZ_p}\im(d)\bigr) \cong\CC_p\otimes_{\ZZ_p}P
\end{multline*}
where the first and third maps are obtained by choosing $\CC_p[G]$-equivariant splittings of the tautological surjections $\CC_p\otimes_{\ZZ_p}P \to \CC_p\otimes_{\ZZ_p}\im(d)$ and $\CC_p\otimes_{\ZZ_p}P \to \CC_p\otimes_{\ZZ_p}H^{a+1}(P^\bullet)$ respectively and the second is  $\bigl((\CC_p\otimes_{\ZZ_p}H^{a+1}(\iota))^{-1}\circ \lambda\circ (\CC_p\otimes_{\ZZ_p}H^a(\iota)),{\id}\bigr)$.

We now write $K_0\bigl(\ZZ_p[G],\CC_p[G]\bigr)$ for the relative algebraic $K$-group of the inclusion $\ZZ_p[G]\subset \CC_p[G]$ and recall that this group is  generated by elements of the form $[Q,\mu,Q']$ where $Q$ and $Q'$ are finitely generated projective $\ZZ_p[G]$-modules and $\mu$ is an isomorphism of $\CC_p[G]$-modules $\CC_p\otimes_{\ZZ_p}Q \cong \CC_p\otimes_{\ZZ_p}Q'$.

In particular, in terms of this description, it can be shown that the element of $K_0\bigl(\ZZ_p[G],\CC_p[G]\bigr)$ obtained by setting
\[ \chi_{G,p}(C,\lambda) := (-1)^a\,[P,\lambda_{P^\bullet},P]\]
is independent of the choice of isomorphism $\iota$ (and hence of the module $P$) and of the splittings made when defining the isomorphism $\lambda_{P^\bullet}$.
This element will be referred to as the `refined Euler characteristic' of the pair $(C,\lambda)$ in the sequel.

\subsection{Cohomology with finite support}\label{cfs} We now turn to consider the finite support cohomology complex of the $p$-adic Tate module $T_p(A)$ of $A^t$.

To do this we write $\Sigma_k(F)$ for the set of $k$-embeddings $F \to k^c$ and $Y_{F/k,p}$ for the module
 $\prod_{\Sigma_k(F)}\ZZ_p$, endowed with the obvious action of $G\times G_k$.

Then the $p$-adic Tate module of the base change of $A^t$ through $F/k$ is equal to
\[ T_{p,F}(A) := Y_{F/k,p}\otimes_{\ZZ_p}T_p(A),\]
where $G$ acts naturally on the first factor and $G_k$ acts diagonally. We set $V_{p,F}(A) := \QQ_p\,T_{p,F}(A)$.

For each place $v$ of $k$ we set $G_{v} := G_{k_v}$ and write $I_{v}$ for the inertia subgroup of $G_{v}$. We also fix a place $w$ of $F$ above $v$ and a corresponding embedding $F\to k_v^c$ and write $\overline{G}_v$ and $\overline{I}_v$ for the images of $G_{v}$ and $I_{v}$ under the induced homomorphism $G_{v} \to G$.

For each non-archimedean place $v$ we then define a complex of $\ZZ_p[G]$-modules by setting
\[
R\Gamma_f\bigl(k_v,T_{p,F}(A)\bigr) :=
\begin{cases}
  T_{p,F}(A)^{I_{v}} \xrightarrow{1- \Fr_v^{-1}} T_{p,F}(A)^{I_{v}},    &\text{ if $v \nmid p$}\\
  H^1_f\bigl(k_v,T_{p,F}(A)\bigr)[-1],                     &\text{ if $v \mid p$}
\end{cases}
\]
where, if $v \nmid p$, the first occurrence of $T_{p,F}(A)^{I_{v}}$ is placed in degree zero and $\Fr_v$ is the natural Frobenius in $G_{v}/I_{v}$ and, if $v \mid p$, then $H^1_f\bigl(k_v,T_{p,F}(A)\bigr)$ denotes the full pre-image in $H^1\bigl(k_v,T_{p,F}(A)\bigr)$ of the `finite part' subspace $H^1_f\bigl(k_v,V_{p,F}(A)\bigr)$ of $H^1\bigl(k_v,V_{p,F}(A)\bigr)$ that is defined by Bloch and Kato in~\cite[(3.7.2)]{bk}.

We next define $R\Gamma_{/f}\bigl(k_v,T_{p,F}(A)\bigr)$ to be a complex of $\ZZ _p[G]$-modules which lies in an exact triangle in $D\bigl(\ZZ _p[G]\bigr)$ of the form
\begin{equation}\label{localtriangle}
 \xymatrix{ R\Gamma_f\bigl(k_v,T_{p,F}(A)\bigr)\ar[r]^{\varpi_{v}} &
            R\Gamma\bigl(k_v,T_{p,F}(A)\bigr)\ar[r]^{\varpi'_{v}} &
            R\Gamma_{/f}\bigl(k_v,T_{p,F}(A)\bigr) \ar[r] & }
\end{equation}
where $\varpi_{v}$ is the natural inflation morphism if $v \nmid p$, and is induced by the inclusion $H^1_f\bigl(k_v,T_{p,F}(A)\bigr) \subseteq H^1\bigl(k_v,T_{p,F}(A)\bigr)$ and the fact that $R\Gamma\bigl(k_v,T_{p,F}(A)\bigr)$ is acyclic in degrees less than one, if $v \mid p$.
We finally set $S := S_{\rm r}\cup S_{\rm b}$, write $\mathcal{O}_{k,S}$ for the subring of $k$ comprising elements that are integral at all non-archimedean places outside $S$ and define the complex $R\Gamma_f\bigl(k,T_{p,F}(A)\bigr)$ so that it lies in an exact triangle in $D\bigl(\ZZ_p[G]\bigr)$ of the form
\begin{equation}\label{rgft}
 \xymatrix@C-2ex{
 R\Gamma_f\bigl(k,T_{p,F}(A)\bigr) \ar[r] &
 R\Gamma\Bigl(\mathcal{O}_{k,S}\bigl[\tfrac{1}{p}\bigr],T_{p,F}(A)\Bigr) \ar[r]^(0.45){\varpi'} &
 \bigoplus\limits_{ v\in S\cup S_p} R\Gamma_{/f}\bigl(k_v, T_{p,F}(A)\bigr)\ar[r] & }.
\end{equation}
Here each $v$-component of $\varpi'$ is equal to the composite of the natural localisation morphism $R\Gamma\bigl(\mathcal{O}_{k,S}[\tfrac{1}{p}],T_{p,F}(A)\bigr) \to R\Gamma\bigl(k_v, T_{p,F}(A)\bigr)$ in \'etale cohomology and the morphism $\varpi'_{v}$.

In the sequel we also set 
\[H^i_f\bigl(k,T_{p,F}(A)\bigr) := H^i\bigl(R\Gamma_f(k ,T_{p,F}(A))\bigr)\]
in each degree $i$.

\begin{lemma}\label{canisos}\
  \begin{enumerate}
  \item[(i)]\label{canisos_i}
      Assume that $A$ and $F$ satisfy the hypotheses~\ref{hyp_c}--\ref{hyp_f}. Then $R\Gamma_f\bigl(k ,T_{p,F}(A)\bigr)$ belongs to $D^{\rm p}\bigl(\ZZ _p[G]\bigr)$. In addition, for any non-archimedean place $v$ of $k$ which ramifies in $F/k$ and does not divide $p$, the complex $R\Gamma_f\bigl(k_v,T_{p,F}(A)\bigr)$ is acyclic.
  \item[(ii)]\label{canisos_ii}
      Assume that $A$ and $F$ satisfy the hypotheses~\ref{hyp_a}, \ref{hyp_b}, \ref{hyp_e} and~\ref{hyp_g}. Then the complex $R\Gamma_f\bigl(k ,T_{p,F}(A)\bigr)$ is acyclic outside degrees one and two and there are canonical identifications of  $H^1_f\bigl(k,T_{p,F}(A)\bigr)$ and $H^2_f\bigl(k,T_{p,F}(A)\bigr)$ with $A^t(F)_p$ and $\Sel_p(A_{F})^\vee$ respectively.
 \end{enumerate}
\end{lemma}
\begin{proof}
 First, since $p$ is odd and $T_{p,F}(A)$ is a (finitely generated) free $\ZZ_p[G]$-module the complexes $R\Gamma\bigl(\mathcal{O}_{k,S}\bigl[\tfrac{1}{p}\bigr],T_{p,F}(A)\bigr)$ and $R\Gamma\bigl(k_v,T_{p,F}(A)\bigr)$ for each place $v$ of $k$ belong to $D^{\rm p}\bigl(\ZZ _p[G]\bigr)$.
In view of the exact triangles~\eqref{localtriangle} and~\eqref{rgft}, claim (i) will therefore follow if we can show that for each $v$ in $S\cup S_p$ the complex $R\Gamma_f\bigl(k_v,T_{p,F}(A)\bigr)$ belongs to $D^{\rm p}\bigl(\ZZ _p[G]\bigr)$, or equivalently (as $P$ is a Sylow $p$-subgroup) that it belongs to $D^{\rm p}\bigl(\ZZ _p[P]\bigr)$.

We first consider the case $v\nmid p$. In this case $R\Gamma_f\bigl(k_v,T_{p,F}(A)\bigr)$ is isomorphic in $D\bigl(\ZZ_p[P]\bigr)$ to $R\Gamma_f\bigl(K_{w'},T_p\bigr)$ where we write $w'$ for a place in $K$ below $w$ and set $T_p: = \ZZ_p[P]\otimes_{\ZZ_p}T_p(A_{/K})$.

In particular, even without assuming our hypotheses, if the ramification index of $v$ in $F/k$ is not divisible by $p$, then $w'$ is unramified in $F/K$ so $T_p^{I_{w'}} \cong\ZZ_p[P]\otimes_{\ZZ_p} T_p(A)^{I_{w'}}$ is a free $\ZZ_p[P]$-module and so $R\Gamma_f(K_{w'},T_p)$ belongs to $D^{\rm p}\bigl(\ZZ _p[P]\bigr)$.

Next we show that hypotheses~\ref{hyp_e} and~\ref{hyp_f} imply that if $w'$ ramifies in $F/k$ and $v \nmid p$, then $R\Gamma_f(K_{w'},T_p)$ (and therefore also $R\Gamma_f\bigl(k_v,T_{p,F}(A)\bigr)$) is acyclic. We write $d^0$ for the differential in degree $0$ of $R\Gamma_f(K_{w'},T_p)$. Then it is clear that $\ker(d^0)=T_p^{G_{w'}}$ vanishes and so it suffices to prove $\cok(d^0)$ vanishes or equivalently, since $P$ and $\cok(d^0)$ are finite $p$-groups, that $\cok(d^0)_P$ vanishes.
But~\ref{hyp_e} implies $T_p^{I_{w'}} = \ZZ_p[P/\overline{I}_{w'}]\otimes_{\ZZ_p}T_p(A_{/K})$ and so the $\ZZ_p$-module $\cok(d^0)_P$ is isomorphic to the cokernel of
the action of $1-\Fr_{w'}^{-1}$ on $T_p(A_{/K})$ and hence of cardinality the maximum power of $p$ that divides
$\det_{\ZZ_p}\bigl(1-\Fr_{w'}^{-1}\bigl\vert T_p(A_{/K})\bigr)$ in $\ZZ_p$. The module $\cok(d^0)_P$ therefore vanishes since the latter determinant is equal to
$\vert A^t(\kappa_{w'})\vert / \vert \kappa_{w'}\vert ^{d}$ and this is a unit at $p$ by~\ref{hyp_f}.

Finally we fix a $p$-adic place $v$ of $k$ and recall from~\cite[after (3.2)]{bk} that the group $H^1_f\bigl(k_v,T_{p,F}(A)\bigr)$ is the image in
$H^1\bigl(k_v,T_{p,F}(A)\bigr) \cong \ZZ_p[G]\otimes_{\ZZ_p[G_w]}H^1\bigl(F_w,T_p(A)\bigr)$ of $\ZZ_p[G]\otimes_{\ZZ_p[G_w]} A^t(F_w)^{\wedge}_p$ under the natural
(injective) Kummer map. Since the group $H^1_f\bigl(k_v,T_{p,F}(A)\bigr)$ is finitely generated over $\ZZ_p$ it is thus enough to show that each module $A^t(F_w)^{\wedge}_p$ is
cohomologically-trivial over $G_w$.

This is clear if the order of $G_w$ is prime to $p$ and true in any other case provided that $A^t(F_w)^{\wedge}_p$ is cohomologically-trivial over each subgroup $C$ of
$G_w$ that has order $p$. 
The point here is that, for a given $p$-Sylow subgroup $P_w$ of $G_w$,
there exists a normal subgroup $C$ of $P_w$ which has order $p$ and hence also a
Hochschild-Serre spectral sequence in Tate cohomology
$\hat H^{a}\bigl(P_w/C,\hat H^b(C,A^t(F_w)^{\wedge}_p)\bigr) \Longrightarrow \hat H^{a+b}\bigl(P_w,A^t(F_w)^{\wedge}_p\bigr)$. Thus if $A^t(F_w)^{\wedge}_p$ is
a cohomologically-trivial $C$-module, then $\hat H^{m}\bigl(P_w,A^t(F_w)^{\wedge}_p\bigr)$ is trivial for all integers $m$, and then~\cite[Chapter VI, (8.7) and (8.8)]{brown}
combine to imply that $A^t(F_w)^{\wedge}_p$ is a cohomologically-trivial $G_w$-module.

We hence fix a subgroup $C$ of $G_w$ of order $p$. Now cohomology over $C$ is periodic of order $2$ and $\QQ_p\cdot A^t(F_w)^{\wedge}_p$ is isomorphic
(via the formal group logarithm)
to the free $\QQ_p[G_w]$-module $F^{d}_w$, so \cite[Corollary to Proposition 11]{casselsfrohlich} implies that the Herbrand quotient $h(C,A^t(F_w)^\wedge_p)$ is equal to 1.

It is thus enough to prove that the group $\hat H^0\bigl(C,A^t(F_w)^{\wedge}_p\bigr)$ vanishes, or equivalently that the natural norm map
$A^t(F_w)^{\wedge}_p \to A^t(F_w^C)^{\wedge}_p$ is surjective.
But the hypotheses~\ref{hyp_c} and~\ref{hyp_d} imply that $A^t_{/F^C}$ has good reduction at all $p$-adic places and further that at any $p$-adic place $w'$ of $F^C$ which
ramifies in $F/F^C$ the reduction is ordinary and such that $A^t(\kappa_{w'})$ has no $p$-torsion, and so the argument of Mazur in~\cite[\S4]{m} implies that the norm map
$A^t(F_w)^{\wedge}_p \to A^t(F_w^C)^{\wedge}_p$ is surjective, as required to complete the proof of claim (i).
(We note in passing that if $p$ is unramified in $F/\QQ$, as will be the case in applications of Lemma \ref{canisos} in the present article, then the relevant
argument is given by \cite[Corollary 4.4]{m}.)

Turning to claim (ii) we note that, under the hypothesis~\ref{hyp_g}, the  argument of~\cite[p. 86-87]{bufl95} implies directly that $H^i_f\bigl(k,T_{p,F}(A)\bigr)$
vanishes if $i \notin \{1,2,3\}$, that $H^3_f\bigl(k,T_{p,F}(A)\bigr)$ is isomorphic to $(A(F)_{p,{\rm tor}})^\vee$ 
and that there is a canonical exact sequence of $\ZZ_p[G]$-modules
\begin{equation*}\xymatrix@R-2.2ex@C-1.1ex{
0 \ar[r] & H^1_f\bigl(k,T_{p,F}(A)\bigr) \ar[r] & A^t(F)_p \ar[r] &
{\bigoplus\limits_{v\in S_{\rm r}\cup S_{\rm b}}} H^0\Bigl(k_v, H^1\bigl(I_v, T_{p,F}(A)\bigr)_{\rm tor}\Bigr)
\ar `r[d] `[ll] `l[dlll] `d[dll] [dll] \\
& H^2_f\bigl(k,T_{p,F}(A)\bigr) \ar[r] & \Sel_p(A_{F})^\vee\ar[r] & 0.}
\end{equation*}
Now hypothesis~\ref{hyp_a} implies that the $p$-group $P$ acts on $A(F)[p]$ with a single fixed point and hence that $H^3_f\bigl(k,T_{p,F}(A)\bigr)$ vanishes.
Claim (ii) will therefore follow if we can show that, under the stated hypotheses, the group $H^0\bigl(k_v, H^1(I_v, T_{p,F}(A))_{\rm tor}\bigr)$ vanishes for all
$v\in S_{\rm r}\cup S_{\rm b}$. In view of the natural isomorphism
\[
 H^0\Bigl(k_v, H^1\bigl(I_v, T_{p,F}(A)\bigr)_{\rm tor}\Bigr)\cong
 \ZZ_p[G]\otimes_{\ZZ_p[G_w]} H^0\Bigl(F_w, H^1\bigl(I_w, T_p(A)\bigr)_{\rm tor}\Bigr)
\]
it is thus enough to show that $H^0\bigl(F_w, H^1(I_w, T_p(A))_{\rm tor}\bigr)$ vanishes for all $v\in S_{\rm r}\cup S_{\rm b}$. On the one hand, if $v \notin S_{\rm b}$, then $H^1\bigl(I_w, T_p(A)\bigr)=\Hom_{\rm cont}\bigl(I_w,T_p(A)\bigr)$ is torsion-free and so the result is clear. On the other hand, if $v \in S_{\rm b}$, then~\ref{hyp_e} implies that $v \notin S_{\rm r}$ so $I_w = I_{w'}$ and hence the group
 $$ H^0\Bigl(F_w, H^1\bigl(I_w, T_p(A)\bigr)_{\rm tor}\Bigr)^{\overline{G}_{w'}} = H^0\Bigl(K_{w'}, H^1\bigl(I_{w'}, T_p(A)\bigr)_{\rm  tor}\Bigr)$$
vanishes as a consequence of~\ref{hyp_b} and the fact that $H^0\Bigl(K_{w'}, H^1\bigl(I_{w'}, T_p(A)\bigr)_{\rm tor}\Bigr)$ has cardinality equal to the maximal power of $p$ that
divides the Tamagawa number of $A^t_{/K}$ at $w'$. Since $\overline{G}_{w'}$ is a $p$-group, this 
implies that $H^0\bigl(F_w, H^1(I_w, T_p(A))_{\rm tor}\bigr)$ vanishes, as required to complete the proof of claim (ii).
\end{proof}

In the sequel we assume that $A$,  $F$ and $p$ satisfy the hypotheses~\ref{hyp_a}--\ref{hyp_g}. Then Lemma~\ref{canisos} implies that the complex
\[ C^{f,\bullet}_{A,F}:= R\Gamma_f\bigl(k ,T_{p,F}(A)\bigr)\]
belongs to $D^{\rm p}\bigl(\ZZ_p[G]\bigr)$, is acyclic outside degrees one and two and is such that for each isomorphism of fields $j: \CC\cong \CC_p$ there exists a canonical composite isomorphism of $\CC_p[G]$-modules of the form
\begin{multline*}
\lambda^{{\rm NT},j}_{A,F}\colon \CC_p\otimes_{\ZZ_p} H^1\bigl(C^{f,\bullet}_{A,F}\bigr) \cong \CC_p\otimes_{\ZZ_p}A^t(F)_p \\\cong \CC_p\otimes_{\CC,j}(\CC\otimes_\ZZ A^t(F))
 \cong \CC_p\otimes_{\CC,j}\Hom_{\CC}\bigl(\CC\otimes_{\ZZ}A(F),\CC\bigr)\\ \cong \CC_p\otimes_{\ZZ_p}\Hom_{\ZZ_p}\bigl(A(F)_p,\ZZ_p\bigr) \cong \CC_p\otimes_{\ZZ_p}H^2\bigl(C^{f,\bullet}_{A,F}\bigr)\end{multline*}
in which the central isomorphism is induced by the $\CC$-linear extension to $\CC\otimes_{\ZZ} A(F)$ of the N\'eron-Tate height of $A$ relative to the field $F$.

Hypothesis~\ref{hyp_a} implies that the $p$-group $P$ acts on $A(F)[p]$ with a single fixed point and hence that the module $A^t(F)_p\cong H^1(C^{f,\bullet}_{A,F})$ is
$\ZZ_p$-free and so the above observations imply that the construction in \S\ref{ec} applies to the pair $\bigl(C^{f,\bullet}_{A,F},\lambda^{{\rm NT},j}_{A,F}\bigr)$ to
give a canonical Euler characteristic
\[ \chi_j(A,F/k) := -\chi_{G,p}\bigl(C^{f,\bullet}_{A,F},\lambda^{{\rm NT},j}_{A,F}\bigr)\]
in the relative algebraic $K$-group $K_0\bigl(\ZZ_p[G],\CC_p[G]\bigr)$.

This element $\chi_j(A,F/k)$ encodes detailed information about a range of aspects of the arithmetic of $A_{/F}$ and in the remainder of~\S\ref{bloch-kato} we shall explicitly relate it to the equivariant Tamagawa numbers that are defined in~\cite{bufl99}. 
In~\S\ref{congruences section} this comparison result plays a key role in the formulation of a precise conjectural description of a pre-image of $\chi_j(A,F/k)$ under the boundary homomorphism  $K_1\bigl(\CC_p[G]\bigr) \to K_0\bigl(\ZZ_p[G],\CC_p[G]\bigr)$.

\subsection{The element \texorpdfstring{$R\Omega_j\bigl(h^1(A_{/F})(1),\ZZ[G]\bigr)$}{ROmega}.}\label{exp structures}

For each isomorphism $j:\CC\cong\CC_p$ as above there is an induced composite homomorphism of abelian groups
\begin{equation}\label{induced rkt} j_{G,*}: K_0\bigl(\ZZ[G],\RR[G]\bigr) \to K_0\bigl(\ZZ[G],\CC[G]\bigr)\cong K_0\bigl(\ZZ[G],\CC_p[G]\bigr) \to K_0\bigl(\ZZ_p[G],\CC_p[G]\bigr)\end{equation}
(where the first and third arrows are induced by the inclusions $\RR[G]\subset \CC[G]$ and $\ZZ[G]\subset \ZZ_p[G]$ respectively).

For each such $j$ we set
\[ R\Omega_j\bigl(h^1(A_{/F})(1),\ZZ[G]\bigr) := j_{G,*} \Bigl(R\Omega\bigl(h^1(A_{/F})(1),\ZZ[G]\bigr)\Bigr)\]
where $R\Omega\bigl(h^1(A_{/F})(1),\ZZ[G]\bigr)$ is the `algebraic part' of the equivariant Tamagawa number for the pair $\bigl(h^1(A_{/F})(1),\ZZ[G]\bigr)$, as defined
(unconditionally under the assumed validity of hypothesis~\ref{hyp_g}) in~\cite[\S3.4]{bufl99}.

In order to relate this element to the Euler characteristic
$\chi_j(A,F/k)$ defined above we must first introduce an auxiliary element of $K_0\bigl(\ZZ_p[G],\CC_p[G]\bigr)$.

To do this we define a complex
\[
C^{{\rm loc},\bullet}_{A,F} := \bigoplus_{v \in S_\infty}R\Gamma\bigl(k_v,T_{p,F}(A)\bigr) \oplus \bigoplus_{v \in S\cup S_p}R\Gamma_f\bigl(k_v,T_{p,F}(A)\bigr),
\]
where, as before, $S$ denotes $S_{\rm r}\cup S_{\rm b}$. It is then clear that $C^{{\rm loc},\bullet}_{A,F}$ is acyclic outside degrees zero and one, that $H^0\bigl(C^{{\rm loc},\bullet}_{A,F}\bigr) = \bigoplus_{v \in S_\infty}H^0(k_v,T_{p,F}(A))$ and that there is a canonical identification of $\QQ_p\otimes_{\ZZ_p}H^1\bigl(C^{{\rm loc},\bullet}_{A,F}\bigr)$ with $\QQ_p\otimes_{\ZZ_p}A^t(F_p)^\wedge_p$,
where we set $F_p:= F\otimes_\QQ\QQ_p$.

We next use this description of the cohomology of $C^{{\rm loc},\bullet}_{A,F}$ to define a canonical isomorphism of $\CC_p[G]$-modules
\[ \lambda_{A,F}^{{\rm exp},j} : \CC_p\otimes_{\ZZ_p}H^0\bigl(C^{{\rm loc},\bullet}_{A,F}\bigr) \cong \CC_p\otimes_{\ZZ_p}H^1\bigl(C^{{\rm loc},\bullet}_{A,F}\bigr).\]
For this we write $\Sigma(k)$ for the set of embeddings $k \to \CC$ and $\Sigma_\sigma(F)$, for each $\sigma$ in $\Sigma(k)$, for the set of embeddings $F \to \CC$ that extend $\sigma$. For $v$ in $S_\infty$ we fix a corresponding element $\sigma_v$ of $\Sigma(k)$ and consider the $\CC$-linear map
\begin{multline*}
 \pi_v\colon\ \CC\otimes_\ZZ H_1\bigl(\sigma_v(A^t)(\CC),\ZZ\bigr)\\
  \longrightarrow \Hom_\CC \Bigl(\CC\otimes_{k,\sigma_v} H^0\bigl(A^t,\Omega^1_{A^t}\bigr),\,\CC\Bigr) = \CC\otimes_{k,\sigma_v}\Hom_k\bigl(H^0(A^t,\Omega^1_{A^t}),k\bigr)
\end{multline*}
that sends the image $\gamma$ in $H_1\bigl(\sigma_v(A^t)(\CC),\ZZ\bigr)$ of a cycle $\hat \gamma$ to the map induced by sending a differential $\omega$ to $\int_{\gamma}\omega := \int_{\hat\gamma}\omega$.

For each place $v$ in $S_\infty$ we also write $Y_{v,F}$ for the module $\prod_{\Sigma_{\sigma_v}(F)}\ZZ$, endowed with its natural action of the direct product $G\times G_v$, and then define $\lambda_{A,F}^{{\rm exp},j}$ to be the following composite isomorphism of $\CC_p[G]$-modules
\begin{align*}
\CC_p\otimes_{\ZZ_p}H^0\bigl(C^{{\rm loc},\bullet}_{A,F}\bigr) = &\,\CC_p\otimes_{\ZZ_p}\bigoplus_{v \in S_\infty} H^0\bigl(k_v,T_{p,F}(A)\bigr) \\
\cong &\, \CC_p\otimes_{\ZZ} \bigoplus_{v \in S_\infty} H^0\Bigl(k_v, Y_{v,F}\otimes_{\ZZ}H_1(\sigma_v(A^t)(\CC),\ZZ)\Bigr)\\
\cong &\,\CC_p\otimes_{\QQ} \Bigl(F\otimes_k \Hom_k\bigl(H^0(A^t,\Omega^1_{A^t}),k\bigr)\Bigr)\\
\cong &\, \CC_p\otimes_{\ZZ_p} A^t(F_p)^{\wedge}_p \\
= &\,\CC_p\otimes_{\ZZ_p}H^1\bigl(C^{{\rm loc},\bullet}_{A,F}\bigr).
\end{align*}
Here the first isomorphism is induced by the canonical comparison isomorphisms
\begin{equation}\label{correct comp}
 T_{p,F}(A) = Y_{F/k,p}\otimes_{\ZZ_p}T_{p}(A)\cong Y_{v,F,p}\otimes_{\ZZ_p} H_1\bigl(\sigma_v(A^t)(\CC),\ZZ\bigr)_p,
 \end{equation}
the second by the maps $\CC_p\otimes_{\CC,j}\pi_v$ and the canonical decomposition
\begin{align}\label{period decomp}
 \RR\otimes_{\QQ}\Bigl(F\otimes_k& \Hom_k\bigl(H^0(A^t,\Omega^1_{A^t}),k\bigr)\Bigr) \\
\cong &\bigoplus_{v\in S_\infty} H^0\Bigl(k_v, \bigl(\CC\otimes_{k,\sigma}F\bigr)\otimes_\CC\bigl(\CC\otimes_{k,\sigma}\Hom_k(H^0(A^t,\Omega^1_{A^t}),k)\bigr)\Bigr)\notag\\
\cong &\bigoplus_{v\in S_\infty} H^0\Bigl(k_v,\, \CC Y_{v,F}\otimes_\CC\bigl(\CC\otimes_{k,\sigma} \Hom_k\bigl(H^0(A^t,\Omega^1_{A^t}),k\bigr)\bigr)\Bigr),\notag
\end{align}
and the third by the sum over places $v$ in $S_p$ of the classical exponential map 
\begin{equation}\label{classical exp}
 F_v\otimes_k \Hom_k\bigl(H^0(A^t,\Omega^1_{A^t}),k\bigr) \cong \Hom_{F_v}\bigl(H^0(A^t_{/F_v},\Omega^1_{A^t_{/F_v}}),F_v\bigr) \cong \QQ_pA^t(F_v)^{\wedge}_p
\end{equation}
where we set $F_v := F\otimes_kk_v$.

Under our stated hypotheses the complex $C^{{\rm loc},\bullet}_{A,F}$ belongs to $D^{\rm p}\bigl(\ZZ_p[G]\bigr)$. It is also acyclic outside degrees zero and one and such that $H^0\bigl(C^{{\rm loc},\bullet}_{A,F}\bigr)$ is $\ZZ_p$-free and so we may use the construction of \S\ref{ec} to define an element of $K_0\bigl(\ZZ_p[G],\CC_p[G]\bigr)$ by setting
\[ \chi_j^{\rm loc}(A,F/k) := \chi_{G,p}\Bigl(C^{{\rm loc},\bullet}_{A,F},\, \lambda_{A,F}^{{\rm exp},j}\Bigr).\]

We are now ready to state the main result of this section. In this result, and the sequel, we shall use the composite homomorphism
\[ \delta_{G,p}:\zeta\bigl(\CC_p[G]\bigr)^{\times}\to K_1\bigl(\CC_p[G]\bigr)\to K_0\bigl(\ZZ_p[G],\CC_p[G]\bigr)\]
where the first arrow is the inverse of the (bijective) reduced norm homomorphism $\Nrd_{\CC_p[G]}$ and the second is the standard boundary homomorphism $\partial_{G,p}$.

\begin{proposition}\label{comparison lemma}
 Assume that $A$ and $F$ satisfy the hypotheses~\ref{hyp_a}--\ref{hyp_g}. Then in $K_0\bigl(\ZZ_p[G],\CC_p[G]\bigr)$ one has
\[ R\Omega_j\bigl(h^1(A_{/F})(1),\ZZ[G]\bigr) = -\chi_j(A,F/k) - \chi_j^{\rm loc}(A,F/k) + \sum_{v \in S\cup S_p}\delta_{G,p}\bigl(L_v(A,F/k)\bigr).\]
Here for each place $v$ in $S\cup S_p$ we set
\[ L_v(A,F/k) := \begin{cases} \Nrd_{\QQ_p[G]}\bigl(1- \Fr_v^{-1}\bigm\vert V_{p,F}(A)^{I_{v}}\bigr), &\text{if $v\nmid p$}\\
\Nrd_{\QQ_p[G]}\bigl(1- \varphi_v\bigm\vert D_{{\rm cr},v}(V_{p,F}(A))\bigr), &\text{if $v\mid p$,}\end{cases}\]
where $\varphi_v$ is the crystalline Frobenius at $v$.
\end{proposition}

\begin{proof}
 To discuss $R\Omega_j\bigl(h^1(A_{/F})(1),\ZZ[G]\bigr)$ we first recall relevant facts concerning the formalism of virtual objects introduced by Deligne in~\cite{delignedet} (for more details see~\cite[\S2]{bufl99}).

With $R$ denoting either $\ZZ_p[G]$ or $\CC_p[G]$ we write $V(R)$ for the category of virtual objects over $R$ and $[C]_R$ for the object of $V(R)$ associated to each $C$ in $D^{\rm p}(R)$. Then $V(R)$ is a Picard category with $\pi_1(V(R))$ naturally isomorphic to $K_1(R)$ (see~\cite[(2)]{bufl99}) and we write $(X,Y) \mapsto X\cdot Y$ for its product and $\eins_R$ for the unit object $[0]_R$. Writing $\mathcal{P}_0$ for the
Picard category with unique object $\eins_{\mathcal{P}_0}$
and $\Aut_{\mathcal{P}_0}(\eins_{\mathcal{P}_0}) = 0$ we use the isomorphism of abelian groups
\begin{equation}\label{virtual iso}
 \pi_0\Bigl(V\bigl(\ZZ_p[G]\bigr)\times_{V(\CC_p[G])}\mathcal{P}_0\Bigr)\cong K_0\bigl(\ZZ_p[G],\CC_p[G]\bigr)
\end{equation}
that is described in~\cite[Proposition 2.5]{bufl99}. In particular, via this isomorphism, each pair comprising an object $X$ of $V\bigl(\ZZ_p[G]\bigr)$ and a morphism $\iota: \CC_p[G]\otimes_{\ZZ_p[G]}X \to \eins_{\CC_p[G]}$ gives rise to a canonical element $[X,\iota]$ of $K_0\bigl(\ZZ_p[G],\CC_p[G]\bigr)$.

Now if $C$ belongs to $D^{\rm p}\bigl(\ZZ_p[G]\bigr)$ and is acyclic outside degrees $a$ and $a+1$ (for any integer $a$), then any isomorphism  $\lambda: H^a\bigl(\CC_p\otimes_{\ZZ_p}C\bigr) \cong H^{a+1}\bigl(\CC_p\otimes_{\ZZ_p}C\bigr)$ of $\CC_p[G]$-modules gives a canonical morphism $\lambda_{\rm Triv}: [\CC_p\otimes_{\ZZ_p}C]_{\CC_p[G]} \to \eins_{\CC_p[G]}$ in $V\bigl(\CC_p[G]\bigr)$.
The associated element $\bigl[[C]_{\ZZ_p[G]},\lambda_{{\rm Triv}}\bigr]$ of  $K_0\bigl(\ZZ_p[G],\CC_p[G]\bigr)$ coincides with the Euler characteristic $\chi_{\ZZ_p[G],\CC_p[G]}\bigl(C,\lambda^{(-1)^a}\bigr)$ defined in~\cite[Definition 5.5]{additivity}. In particular, from Proposition 5.6(3) in loc. cit. it follows that $\bigl[[C[-1]]_{\ZZ_p[G]},\lambda_{{\rm Triv}}\bigr] = -\bigl[[C]_{\ZZ_p[G]},\lambda_{{\rm Triv}}\bigr]$, whilst Theorem 6.2 and Lemma 6.3 in loc. cit. combine to imply that if $H^a(C)$ is $\ZZ_p$-free, then
\begin{equation}\label{normalisation}
  \bigl[[C]_{\ZZ_p[G]},\lambda_{{\rm Triv}}\bigr] = \chi_{G,p}(C,\lambda) + \delta_{G,p}\biggl(\prod_{i \equiv 1,2 \bmod{4} }\Nrd_{\QQ_p[G]} \bigl(-{\id}\bigm\vert \QQ_p\otimes_{\ZZ_p}H^i(C)\bigr)^{(-1)^i}\biggr),\end{equation}
where $\chi_{G,p}(-,-)$ is the explicit Euler characteristic discussed in \S\ref{ec}.

We now set
\[ C^{c,\bullet}_{A,F} := R\Gamma_c\bigl(\mathcal{O}_{k,S}[\tfrac{1}{p}],T_{p,F}(A)\bigr),\]
which is the complex of the compactly supported \'etale cohomology of $T_{p,F}(A)$ on $\Spec\bigl(\mathcal{O}_{k,S}[\tfrac{1}{p}]\bigr)$, regarded as a complex of $\ZZ_p[G]$-modules in the natural way.

Then a direct comparison of the definitions of $C^{c,\bullet}_{A,F}$ and $C^{f,\bullet}_{A,F}$ shows that there is a canonical exact triangle in $D^{\rm p}\bigl(\ZZ_p[G]\bigr)$
\begin{equation}\label{can tri}
 \xymatrix@1{
 C^{{\rm loc},\bullet}_{A,F}[-1] \ar[r] & C^{c,\bullet}_{A,F} \ar[r] &  C^{f,\bullet}_{A,F}\ar[r] &  C^{{\rm loc},\bullet}_{A,F}.}
\end{equation}

We consider the following diagram in $V\bigl(\CC_p[G]\bigr)$
\begin{equation}\label{comm diag} \xymatrix@C+1ex@R+1ex{
\bigl[\CC_p\otimes_{\ZZ_p}C^{c,\bullet}_{A,F}\bigr]_{\CC_p[G]} \ar[r]^(.34){\Delta} \ar[d]_{\alpha} &
    \bigl[\CC_p\otimes_{\ZZ_p}C^{{\rm loc},\bullet}_{A,F}[-1]\bigr]_{\CC_p[G]}\cdot \bigl[\CC_p\otimes_{\ZZ_p}C^{f,\bullet}_{A,F}\bigr]_{\CC_p[G]}
      \ar[d]^{\bigl[(\lambda_{A,F}^{{\rm exp},j})_{\rm Triv}\bigr]_{\CC_p[G]}\cdot \bigl[(\lambda^{{\rm NT},j}_{A,F})_{\rm Triv}\bigr]_{\CC_p[G]}}\\
\eins_{\CC_p[G]} \ar[r]^{\iota} \ar[d]_{\alpha'} & \eins_{\CC_p[G]}\cdot\eins_{\CC_p[G]} \ar[d]^{\bigl(\prod_{v\in S\cup S_p}L_v(A,F/k)\bigr)\cdot{\id}}\\
\eins_{\CC_p[G]} \ar[r]^{\iota} & \eins_{\CC_p[G]}\cdot\eins_{\CC_p[G]}.}
\end{equation}
Here $\Delta$ is the morphism induced by the scalar extension of~\eqref{can tri}, $\iota$ denotes the canonical identification and the morphisms  $\alpha$ and $\alpha'$ are defined by the condition that the two squares commute.

We claim that this definition of $\alpha$ and $\alpha'$ implies that
\begin{equation}\label{comp eq}\alpha'\circ \alpha = (\CC_p\otimes_{\RR,j}\vartheta_\infty)\circ(\CC_p\otimes_{\QQ_p}\vartheta_p)^{-1}\end{equation}
where the morphisms $\vartheta_\infty$ and $\vartheta_p$ are as constructed in~\cite[\S3.4]{bufl99}. To verify this we note that the scalar extension of~\eqref{can tri} is naturally isomorphic to the exact triangle in $D^{\rm p}\bigl(\CC_p[G]\bigr)$ induced by the central column of the diagram~\cite[(26)]{bufl99} and then simply compare the explicit definitions of the morphisms $\lambda_{A,F}^{{\rm exp},j}$ and $\vartheta_p$ and of $\lambda^{{\rm NT},j}_{A,F}$ and $\vartheta_\infty$.
After this it only remains to note the following fact. For each place $v \in S\setminus S_p$, respectively $v \in S_p$, we write $V_{p,v}$ and $\phi_v$ for $V_{p,F}(A)^{I_{v}}$ and $\Fr_v^{-1}$, respectively $D_{{\rm cr},v}\bigl(V_{p,F}(A)\bigr)$ and $\varphi_v$, and then $V_{p,v}^\bullet$ for the complex $V_{p,v}\xrightarrow{1-\phi_v}V_{p,v}$, with the first term placed in degree zero.
Our definition of $\lambda_{A,F}^{{\rm exp},j}$ implicitly uses the morphism $[V^\bullet_{p,v}]_{\QQ_p[G]} \to \eins_{\QQ_p[G]}$ induced by the acyclicity of $V^\bullet_{p,v}$ whereas the definition of $\vartheta_p$ uses (via~\cite[(19) and (22)]{bufl99}) the morphism $[V^\bullet_{p,v}]_{\QQ_p[G]} \to \eins_{\QQ_p[G]}$ induced by the identity map on $V_{p,v}$;
the occurrence of the morphism $\alpha'$ in the equality~\eqref{comp eq} is thus accounted for by applying the remark made immediately after~\cite[(24)]{bufl99} to each of the complexes $V^\bullet_{p,v}$ and noting that $\Nrd_{\QQ_p[G]}(1-\phi_v\mid V_{p,v}) = L_v(A,F/k)$.

Now, taking into account the equality~\eqref{comp eq}, the term $R\Omega_j\bigl(h^1(A_{/F})(1),\ZZ[G]\bigr)$ is defined in~\cite{bufl99} to be equal to
\[
 \Bigl[\bigl[C^{c,\bullet}_{A,F}\bigr]_{\ZZ_p[G]},\, (\CC_p\otimes_{\RR,j}\vartheta_\infty)\circ(\CC_p\otimes_{\QQ_p}\vartheta_p)^{-1}\Bigr] = \Bigl[\bigl[C^{c,\bullet}_{A,F}\bigr]_{\ZZ_p[G]},\,\alpha'\circ \alpha\Bigr].
\]
The product structure of $\pi_0\bigl(V(\ZZ_p[G])\times_{V(\CC_p[G])}\mathcal{P}_0\bigr)$ then combines with the commutativity of~\eqref{comm diag} to imply that it is also equal to

\begin{align*}
  \Bigl[ \bigl[C&^{c,\bullet}_{A,F}  \bigr]_{\ZZ_p[G]}, \,\alpha\Bigr]  + \bigl[\eins_{\ZZ_p[G]},\,\alpha'\bigr]\\
& = \Bigl[ \bigl[C^{{\rm loc},\bullet}_{A,F}[-1]\bigr]_{\ZZ_p[G]},\,\bigl(\lambda_{A,F}^{{\rm exp},j}\bigr)_{\rm Triv}\Bigr]
    + \Bigl[ \bigl[C^{f,\bullet}_{A,F}\bigr]_{\ZZ_p[G]},\,\bigl(\lambda^{{\rm NT},j}_{A,F}\bigr)_{\rm Triv}\Bigr]
    + \bigl[\eins_{\ZZ_p[G]},\,\alpha'\bigr]\\
& = -\Bigl[\bigl[C^{{\rm loc},\bullet}_{A,F}\bigr]_{\ZZ_p[G]},\,\bigl(\lambda_{A,F}^{{\rm exp},j}\bigr)_{\rm Triv}\Bigr]
   + \Bigl[\bigl[C^{f,\bullet}_{A,F}\bigr]_{\ZZ_p[G]},\, \bigl(\lambda^{{\rm NT},j}_{A,F}\bigr)_{\rm Triv}\Bigr] \\
& \phantom{= - aa}   + \Bigl[\eins_{\ZZ_p[G]},\,\prod_{v\in S\cup S_p}L_v(A,F/k)\Bigr].
\end{align*}
This implies the claimed result since~\eqref{normalisation} implies that
\[ \Bigl[\bigl[C^{{\rm loc},\bullet}_{A,F}\bigr]_{\ZZ_p[G]},\,\bigl(\lambda_{A,F}^{{\rm exp},j}\bigr)_{\rm Triv}\Bigr] = \chi_j^{\rm loc}(A,F/k)\]
(as $\delta_{G,p}\bigr(\Nrd_{\QQ_p[G]}\bigl(-{\id}\bigm\vert \QQ_p\otimes_{\ZZ_p}H^1(C^{{\rm loc},\bullet}_{A,F})\bigr)\bigr) = 0$ because $\QQ_p\otimes_{\ZZ_p}H^1\bigl(C^{{\rm loc},\bullet}_{A,F}\bigr)$ is a free $\QQ_p[G]$-module) and
\[ \Bigl[\bigl[C^{f,\bullet}_{A,F}\bigr]_{\ZZ_p[G]},\,\bigl(\lambda^{{\rm NT},j}_{A,F}\bigr)_{\rm Triv}\Bigr] = - \chi_j(A,F/k)\]
(because $\prod_{i=1}^{i=2}\Nrd_{\QQ_p[G]}\bigl(-{\id}\bigm\vert \QQ_p\otimes_{\ZZ_p}H^i(C^{f,\bullet}_{A,F})\bigr)^{(-1)^i} = 1$) whilst
 the explicit description of the isomorphism~\eqref{virtual iso} implies that
\begin{align*} \Bigl[\eins_{\ZZ_p[G]},\,\prod_{v\in S\cup S_p}L_v(A,F/k)\Bigr] =& \delta_{G,p}\Bigl(\prod_{v\in S\cup S_p}L_v(A,F/k)\Bigr)\\ =& \sum_{v\in S\cup S_p}\delta_{G,p}(L_v(A,F/k)).\end{align*}
\end{proof}

\subsection{The local term}\label{local term section}
 In this section we explicitly compute the Euler characteristic $\chi^{\rm loc}_{j}(A,F/k)$ 
 that occurs in Proposition~\ref{comparison lemma} in terms of both archimedean periods and global Galois-Gauss sums.

In the sequel we sometimes suppress explicit reference to the fixed identification of fields $j:\CC \cong \CC_p$. In addition, for each natural number $m$ we will write $[m]$ for the set of integers $i$ which satisfy $1 \le i\le m$.

\subsubsection{Periods and Galois Gauss sums.}\label{archi_subsubsection}

We fix N\'eron models $\mathcal{A}^t$ for $A^t$ over $\mathcal{O}_k$ and $\mathcal{A}^t_v$ for $A^t_{/k_v}$ over $\mathcal{O}_{k_v}$ for each $v$ in $S_p^k$, and then fix a $k$-basis $\{\omega_{b}\}_{b \in [d]}$ of the space of invariant differentials $H^0(A^t, \Omega^1_{A^t})$ which gives $\mathcal{O}_{k_v}$-bases of $H^0\bigl(\mathcal{A}^t_v,\Omega_{\mathcal{A}^t_v}^1\bigr)$ for each $v$ in $S_p^k$ and is also such that each $\omega_b$ extends to an element of $H^0\bigl(\mathcal{A}^t, \Omega^1_{\mathcal{A}^t}\bigr)$.

For each place $v$ in $S_\RR$ we fix $\ZZ$-bases $\{\gamma_{v,a}^+\}_{a\in [d]}$ and $\{\gamma_{v,a}^-\}_{a\in [d]}$ of $H_1\bigl(\sigma_v(A^t)(\CC),\ZZ\bigr)^{c=1}$ and $H_1\bigl(\sigma_v(A^t)(\CC),\ZZ\bigr)^{c=-1}$, where $c$ denotes complex conjugation, and then set
\[
\Omega_{v}^+(A) := \Biggl\vert\det \biggl( \int_{\gamma_{v,a}^{+}} \omega_b\biggr)_{\!a,b} \Biggr\vert, \qquad
\Omega_{v}^-(A) := \Biggl\vert\det \biggl( \int_{\gamma_{v,a}^{-}} \omega_b\biggr)_{\!a,b} \Biggr\vert,
\]
where in both matrices $(a,b)$ runs over $[d]\times [d]$. For each $v$ in $S_\CC$ we fix a $\ZZ$-basis $\{\gamma_{v,a}\}_{a\in [2d]}$ of $H_1\bigl(\sigma_v(A^t)(\CC),\ZZ\bigr)$ and set
\[
 \Omega_v(A) := \Biggl\vert\det \biggl( \int_{\gamma_{v,a}} \omega_b , c\Bigl(\int_{\gamma_{v,a}} \omega_b\Bigr) \biggr)_{\!a,b} \Biggr\vert
\]
where $(a,b)$ runs over $[2d]\times [d]$. (By explicitly computing integrals these periods can be related to those obtained by integrating measures as occurring in the classical formulation of the Birch and Swinnerton-Dyer conjecture -- see, for example, Gross~\cite[p. 224]{G-BSD}).

For any $\psi\in\Ir(G)$ and $v\in S_\RR$, we set $\psi_v^+(1) := \dim_\CC V_\psi^{G_{k_v}}$ and $\psi_v^-(1) := \psi(1) - \psi_v^+(1)$. Then we define the periods
\begin{equation*}
  \Omega_v(A,\psi) :=
   \begin{cases}
     \Omega^+_v(A)^{\psi_v^+(1)}\cdot \Omega^-_v(A)^{\psi_v^-(1)} &\text{ if $v\in S_{\RR}$ and}\\
     \Omega_v(A)^{\psi(1)} &\text{ if $v\in S_{\CC}$;}
   \end{cases}
\end{equation*}
and $\Omega(A,\psi) := \prod_{v \in S_\infty} \Omega_v(A,\psi)$ and finally set
\[
  \Omega(A,F/k) := \sum_{\psi \in \Ir(G)} \Omega(A,\psi)\, e_\psi \in \zeta\bigl(\CC_p[G]\bigr)^\times.
\]
For each $v$ in $S_\RR$, respectively. in $S_\CC$, we also define $w_v(\psi)$ to be equal to $i^{\psi^-_v(1)}$, respectively. $i^{\psi(1)}$, and then set
 $w_\infty(\psi) := \prod_{v \in S_\infty} w_v(\psi)$ and
\[ w_\infty(F/k) := \sum_{\psi\in \Ir(G)} w_\infty(\psi)\,e_\psi \in \zeta\bigl(\CC_p[G]\bigr)^\times.\]

To describe the relevant Galois Gauss sums we first define for each non-archi\-me\-dean place $v$ of $k$ the `non-ramified characteristic' $u_v$ to be the image under the natural induction map $\zeta\bigl(\CC_p[\,\overline{G}_v\,]\bigr)^\times \to \zeta\bigl(\CC_p[G]\bigr)^\times$ of the element $(1-e_{\overline{I}_v}) + (-\Fr_w^{-1})e_{\overline{I}_v}$ of $\zeta(\QQ[\overline{G}_v])^\times$.

For each character $\psi$ in $\Ir(G)$ we then define a
modified equivariant global Galois-Gauss sum by setting
\[
  \tau^*(F/k) := \biggl(\prod_{v \in S_{\rm r}^k} u_v \biggr) \sum_{\psi \in \Ir(G)} e_\psi\, \tau\bigl(\QQ,\Ind_k^\QQ\psi\bigr) \in \zeta\bigl(\QQ^c[G]\bigr)^\times,
\]
where the individual Galois Gauss sums $\tau(\QQ,-)$ are as defined by Martinet in~\cite{martinet}.

\subsubsection{Computation of the local Euler characteristic.}
 The explicit computation of the Euler characteristic $\chi_j^{\rm loc}(A,F/k)$ is made considerably more difficult by the presence of $p$-adic places which ramify in
 any of the relevant field extensions. For simplicity in the sequel, and to focus attention on the key ideas in the present article, we therefore impose the following
 additional hypothesis

\begin{conditions}\setcounter{condone}{7}
  \item\label{hyp_h} $p$ is unramified in $F/\QQ$.
\end{conditions}
A full treatment of the terms $\chi_j^{\rm loc}(A,F/k)$ in the general (ramified) case will then be given in a future article.

In the following result we use the elements $L_v(A,F/k)$ of $\zeta(\QQ_p[G])^\times$ that are defined in Proposition~\ref{comparison lemma}. We will also write $j_*$ for the isomorphism $\zeta(\CC[G])^\times \cong \zeta(\CC_p[G])^\times$ that is induced by $j$. We finally recall that $d$ denotes the dimension of $A$.

\begin{theorem}\label{comp prop}
 If $A$ and $F$ satisfy the hypotheses~\ref{hyp_a}--\ref{hyp_h}, then one has
\[ \chi_j^{\rm loc}(A,F/k) = \delta_{G,p}\Bigl(j_*(\frac{w_\infty(F/k)^{d}\cdot \Omega(A,F/k)}{\tau^*(F/k)^{d}})\prod_{v \in S_{\rm b}\cup S_{p}}L_v(A,F/k)\Bigr).\]
\end{theorem}
\begin{proof}
  We set $H^{F/k}_{A,v}:= H^0\bigl(k_v,Y_{v,F}\otimes_{\ZZ}H_1(\sigma_v(A^t)(\CC),\ZZ)\bigr)$ for each $v$ in $S_\infty$ and then also $H^{F/k}_{A} := \bigoplus_{v \in S_\infty}H^{F/k}_{A,v}$ and $\tilde C^{{\rm loc},\bullet}_{A,F} := H^{F/k}_{A,p}[0] \oplus A^t(F_p)^{\wedge}_p[-1]$. Then, as $p$ is odd and $R\Gamma_f\bigl(k_v,T_{p,F}(A)\bigr)$ is acyclic for all $v$ in $S_{\rm r}$ (by Lemma~\ref{canisos}(i)), the comparison isomorphisms~\eqref{correct comp} induce a natural isomorphism
\[ \kappa: C^{{\rm loc},\bullet}_{A,F} \cong \tilde C^{{\rm loc},\bullet}_{A,F} \oplus \bigoplus_{v \in S_{\rm b}} H^1_f\bigl(k_v,T_{p,F}(A)\bigr)[-1]\]
in $D^{\rm p}\bigl(\ZZ_p[G]\bigr)$ and hence an equality in $K_0\bigl(\ZZ_p[G],\CC_p[G]\bigr)$
\begin{align}\label{first}
    \chi_j^{\rm loc}(A,F/k) -\chi_{G,p}\bigl(\tilde C^{{\rm loc},\bullet}_{A,F},\tilde\lambda^{\rm exp}_{A,F}\bigr) &= \sum_{v\in S_{\rm b}} \chi_{G,p}\Bigl(H^1_f\bigl(k_v,T_{p,F}(A)\bigr)[-1],\,0\Bigr)\\
  &=  \delta_{G,p}\Bigl(\prod_{v\in S_{\rm b}} L_v(A,F/k)\Bigr).\notag\end{align}
Here we set $\tilde\lambda^{\rm exp}_{A,F} := \bigl(\QQ_p\otimes_{\ZZ_p}H^1(\kappa)\bigr)\circ\lambda^{\rm exp}_{A,F}\circ \bigl(\QQ_p\otimes_{\ZZ_p}H^0(\kappa)\bigr)^{-1}$ and the second equality follows by combining for each $v$ in $S_{\rm b}$ the explicit definition of the term $L_v(A,F/k)$ together with the facts that there is an exact sequence of $\ZZ_p[G]$-modules
 \[
  \xymatrix@1@C+3.5ex{  0 \ar[r] & T_{p,F}(A)^{I_{v}} \ar[r]^{1- \Fr_v^{-1}} & T_{p,F}(A)^{I_{v}} \ar[r] & H^1_f\bigl(k_v,T_{p,F}(A)\bigr )\ar[r] & 0},
 \]
and that $T_{p,F}(A)^{I_{v}}$ is a projective $\ZZ_p[G]$-module (as a consequence of hypothesis~\ref{hyp_e}).

We next note that $\tilde\lambda^{\rm exp}_{A,F}:=  \mu_2\circ \mu_1$, where $\mu_1$ is the second isomorphism that occurs in the definition of $\lambda_{A,F}^{{\rm exp},j}$ and $\mu_2 = \bigoplus_{v \in S_p}\mu_2^v$ where each $\mu_2^v$ is induced by the classical exponential map~\eqref{classical exp}.

For each $v$ in $S_p$ we set $\mathcal{O}_{F_v} := \mathcal{O}_F\otimes_{\mathcal{O}_k}\mathcal{O}_{k_v}$. Then hypothesis~\ref{hyp_h} implies that each such place $v$ is unramified in $F_w/k_v$ and hence, by Noether's Theorem, that the $\ZZ_p[G]$-module $\mathcal{O}_{F_v} \cong \prod_{w'\mid v}\mathcal{O}_{F_{w'}}$ is free. In particular, the $\ZZ_p[G]$-module
\[ J_{A,F_p} := \bigoplus_{v \in S_p}\mathcal{O}_{F_v}\otimes_{\mathcal{O}_{k_v}} \Hom_{\mathcal{O}_{k_v}} \Bigl( H^0 \bigl(\mathcal{A}^t_{v},\Omega^1_{\mathcal{A}^t_v}\bigr), \mathcal{O}_{k_v}\Bigr),\]
is free and so in $K_0\bigl(\ZZ_p[G],\CC_p[G]\bigr)$ one has an equality
\begin{multline*}\label{second}
  \chi_{G,p}\bigl(\tilde C^{{\rm loc},\bullet}_{A,F},\tilde\lambda^{\rm exp}_{A,F}\bigr)\\
= \chi_{G,p}\bigl(H^{F/k}_{A,p}[0]\oplus J_{A,F_p}[-1],\mu_1\bigr) +
\chi_{G,p}\bigl(J_{A,F_p}[0]\oplus A^t(F_p)^{\wedge}_p[-1],\mu_2\bigr).
\end{multline*}
Combining this equality with~\eqref{first} and the explicit computations in Lemmas~\ref{fontainemessing} and~\ref{tricky} below one finds that $\chi_j^{\rm loc}(A,F/k)$ is equal to
\[
 \delta_{G,p}\Bigl(j_*(w_\infty(F/k)^{d}\Omega(A,F/k))\prod_{v \in S_{\rm b}\cup S_{p}}  L_v(A,F/k)\Bigr) - d
 \cdot (\mathcal{O}_{F,p},\pi_{F},Y_{F,p}).
\]
 Here we have set $Y_{F} := \prod_{\Sigma(F)}\ZZ$, endowed with its natural action of $G$, and written $\pi_F: \CC_p\otimes_{\QQ_p} F_p \cong \CC_p\otimes_{\ZZ_p}Y_{F,p}$ for the natural isomorphism. To deduce the claimed result from this expression it suffices to show $\bigl(\mathcal{O}_{F,p},\pi_F,Y_{F,p}\bigr)$ is equal to $\delta_{G,p}\bigl(j_*(\tau^*(F/k))\bigr)$ and, since no $p$-adic place of $k$ ramifies in $F/k$, this follows directly from equation~(34), Proposition~7.1 and Corollary~7.6 in~\cite{bleyburns}.
\end{proof}

\begin{lemma} \label{fontainemessing}
 If at each place $v$ in $S_p$  the extension $F_w/\QQ_p$ is unramified and $A$ has good reduction, then one has
$$
\chi_{G,p}\Bigl(J_{A,F_p}[0]\oplus A^t(F_p)^{\wedge}_p[-1],\mu_{2}\Bigr) = \delta_{G,p}\Bigl(\prod_{v \in S_p}L_v(A,F/k)\Bigr).
$$
\end{lemma}
\begin{proof}
 We fix $v$ in $S_p$ and write $\mathcal{A}^t_{v,F}$ for the N\'eron model $\mathcal{A}^t_v\times_{\mathcal{O}_{k_v}}\mathcal{O}_{F_v}$ of $A^t$ over $\mathcal{O}_{F_v} := \prod_{w'\mid v}\mathcal{O}_{F_w'}$.
Then $\Hom_{\mathcal{O}_{F_v}}\bigl(H^0(\mathcal{A}^t_{v,F},\Omega^1_{\mathcal{A}^t_{v,F}}\bigr), \mathcal{O}_{F_v})$ is naturally isomorphic to the free  $\ZZ_p[G]$-module $J_{A,F_v} := \mathcal{O}_{F_v}\otimes_{\mathcal{O}_{k_v}}\Hom_{\mathcal{O}_{k_v}} \bigl(H^0(\mathcal{A}^t_{v},\Omega^1_{\mathcal{A}^t_v}), \mathcal{O}_{k_v}\bigr)$.

In addition, the stated hypotheses on $v$ imply that there exists a full free $\ZZ_p[G]$-submodule 
\[ \mathcal{D}_v := D_{{\rm cr},v}\bigl(T_{p,F}(A)\bigr)\cong\mathcal{O}_{F_v}\otimes_{\mathcal{O}_{k_v}} D_{{\rm cr},v}\bigl(T_p(A)\bigr)\]
of $D_{{\rm cr},v}(V_{p,F}(A))$,
where $D_{{\rm cr},v}$ is the quasi-inverse to the functor of Fontaine and Lafaille used by Niziol in~\cite{N}, and the theory of Fontaine and Messing~\cite{fm} implies that the canonical comparison isomorphism 
\[ \Hom_{F_v}\bigl(H^0(A^t_{F_v},\Omega^1_{A^t_{F_v}}),F_v\bigr) \cong D_{{\rm dR},v}\bigl(V_{p,F}(A)\bigr)/F^0\]
maps $J_{A,F_v} $ to $\mathcal{D}_v/F^0\mathcal{D}_v$ (see \S 5 in \cite{bk}). 

In this case it is also shown in~\cite{bk} that there is a natural short exact sequence of perfect complexes of $\ZZ_p[G]$-modules (with vertical differentials)
\begin{equation*}\label{f-m-seq}
\xymatrix@C+2ex{ 0 \ar[r] &
 \mathcal{D}_v/F^0\mathcal{D}_v \ar[r]^(0.43){(0,\id)} &
 \mathcal{D}_v\oplus\mathcal{D}_v/F^0\mathcal{D}_v \ar[r]^(0.65){(\id,0)} &
 \mathcal{D}_v \ar[r] & 0 \\
 0\ar[r] &
 0\ar[r]\ar[u]&
 \mathcal{D}_v \ar[r]^{\id} \ar[u]^{(1-\varphi_v,\pi)} &
 \mathcal{D}_v \ar[r] \ar[u]^{1 - \varphi_v}  &
0.}
\end{equation*}
Here the term $\mathcal{D}_v/F^0\mathcal{D}_v$ in the first complex occurs in degree $1$ and $\pi$ is the tautological projection. Further, the exponential map of Bloch and Kato
 maps the cohomology in degree $1$ of the second complex bijectively to $H^1_f(k_v,T_{p,F}(A)) = A^t(F_v)^{\wedge}_p$ and the
differential $1-\varphi_v$ of the third complex is injective. For a proof of all these claims see Lemma 4.5 (together with Example 3.11) in loc. cit..
The long exact sequence of cohomology of the above exact sequence thus gives rise to a short exact sequence of $\ZZ_p[G]$-modules $0 \to J_{A,F_v} \to A^t(F_v)^{\wedge}_p \to \cok(1-\varphi_v)\to  0$, in which (following~\cite[Example 3.11]{bk}) the second arrow is equal to $\mu_2^v$. This sequence then implies that the term $\chi_{G,p}\bigl(J_{A,F_v}[0]\oplus A^t(F_v)^{\wedge}_p[-1],\mu^v_{2}\bigr)$ is equal to
\[ \chi_{G,p}\bigl(\cok(1-\varphi_v)[-1],0\bigr) = \partial_{G,p}\bigl((1-\varphi_v\mid D_{{\rm cr},v}(V_{p,F}(A)))\bigr) = \delta_{G,p}\bigl(L_v(A,F/k)\bigr)
\]
and by summing over all places $v$ in $S_p$ this implies the claimed equality.
\end{proof}

For each $\sigma\in \Sigma(k)$ we fix an embedding $\hat \sigma \in \Sigma_\sigma(F)$.

\begin{lemma}\label{tricky}
 In $K_0\bigl(\ZZ_p[G],\CC_p[G]\bigr)$ one has
\begin{multline*}\label{desired period}
  \chi_{G,p} \Bigl(H^{F/k}_{A,p}[0]\oplus J_{A,F_p}[-1],\,\mu_1\Bigr)\\ = -d\cdot(\mathcal{O}_{F,p},\pi_{F},Y_{F,p})+\delta_{G,p}\bigl(j_*(w_\infty(F/k)^d\,\Omega(A,F/k))\bigr).
\end{multline*}
\end{lemma}
\begin{proof}
 For each place $v$ in $S_\RR$ we define $\tau_v$ in $G$ via the equality $\tau_v(\hat\sigma_v) = c\circ \hat\sigma_v$. For each integer $a$ in $[d]$ we then define an element
$$
x_{v,a}:=  \tfrac{1}{2} (1+ \tau_v)\hat\sigma_v\otimes_{\ZZ}\gamma_{v,a}^{+} + \tfrac{1}{2}(1- \tau_v)\hat\sigma_v\otimes_{\ZZ}\gamma_{v,a}^{-}
$$
and we note that, since $p$ is odd, the set $\mathcal{X}_v := \{x_{v,a}\}_{a\in [d]}$ is a $\ZZ_p[G]$-basis of $H^{F/k}_{A,v,p}$.

For $v$ in $S_\CC$ and $a$ in $[2d]$ we set $x_{v,a} := \hat\sigma_v\otimes_\ZZ\gamma_{v,a}$ and note that $\mathcal{X}_v:= \{x_{v,a}\}_{a\in [2d]}$ is a $\ZZ_p[G]$-basis of $H^{F/k}_{A,v,p}$.

We next set $n := [k:\QQ] = |S_\RR| + 2|S_\CC|$ and fix an $\mathcal{O}_{k,p}[G]$-generator $Z$ of $\mathcal{O}_{F_p}$ and a $\ZZ_p$-basis $\{z_m\}_{m \in [n]}$ of $\mathcal{O}_{k,p}$ and set $t_{m,b} := Zz_m\otimes_{k}\omega_b^*$ with $\omega_b^*$ the element of $\Hom_{k} \bigl(H^0(A^t,\Omega^1_{A^t}),k\bigr)$ that is dual to $\omega_b$. Then the set $\mathcal{T}:= \{t_{m,b}\}_{(m,b)\in [n]\times [d]}$ is a $\ZZ_p[G]$-basis of $J_{A,F_p}$.

The key to our argument is to compute the matrix of $\mu_1$ with respect to the bases $\mathcal{T}$ and $\mathcal{X}:= \bigcup_{v\in S_\infty}\mathcal{X}_v$ of $\RR\otimes_{\QQ}\bigl(F\otimes_k \Hom_k(H^0(A^t,\Omega_{A^t}),k)\bigr)$ and $\RR\otimes_\ZZ H^{F/k}_{A}$.
To do this we find it convenient to introduce an auxiliary basis. Thus, for $v$ in $S_\RR$ we set $\mathcal{Y}_v := \{y_{v,b}\}_{b \in [d]}$ with $y_{v,b}:= \hat\sigma_v\otimes_{\ZZ}\sigma_{v,*}(\omega^*_b)$, and for $v$ in $S_\CC$ we set $\mathcal{Y}_v := \{y_{v,b}\}_{b \in [2d]}$ with $y_{v,b}:= \hat\sigma_v\otimes_\ZZ\sigma_{v,*}(\omega^*_b)$ if $b \in [d]$ and $y_{v,b}:= \hat\sigma_v\otimes_\ZZ c\bigl(\sigma_{v,*}(\omega^*_b)\bigr)$ if $b \in [2d]\setminus [d]$.
Then $\mathcal{Y} := \cup_{v \in S_\infty}\mathcal{Y}_v$ is an $\RR[G]$-basis of $\bigoplus_{v \in S_\infty}k_vY_{v,F}\otimes_{k_v}\bigl(k_v\otimes_{k,\sigma_v}\Hom_k(H^0(A^t,\Omega^1_{A^t}),k)\bigr)$ and for each index $a$ one has
\begin{equation*}\label{first period}
 \pi_v(x_{v,a}) =
  \begin{cases}
     \sum_{j =1}^{d}\tfrac{1}{2}(\Omega_{v,a,j}^+ (1+ \tau_v) + \Omega_{v,a,j}^- (1- \tau_v))y_{v,j}, &\text{if $v \in S_\RR$},\\
     \sum_{j =1}^{2d}\Omega_{v,a,j}\,y_{v,j}, &\text{if $v \in S_\CC$}
  \end{cases}
\end{equation*}
with $\Omega_{v,a,j}^{\pm} := \int_{\gamma_{v,a}^{\pm}} \omega_j$ and $\Omega_{v,a,j} := \int_{\gamma_{v,a}} \omega_j$ if $a \in [d]$ and $\Omega_{v,a,j} := c\bigl(\int_{\gamma_{v,a}} \omega_j\bigr)$ if $a \in [2d]\setminus [d]$. This formula implies that the matrix of $\mu_1$ with respect to the bases $\mathcal{X}$ and $\mathcal{T}$ is $M^{-1}\Phi$ where $\Phi = \diag((\Phi_v)_{v \in S_\infty})$ is a diagonal block matrix with blocks
\[
 \Phi_v :=
   \begin{cases}
       \bigl(\Omega_{v,a,j}^+(1+ \tau_v)/2 + \Omega_{v,a,j}^-(1- \tau_v)/2\bigr)_{(a,j)\in [d]\times[d]} &\text{if $v \in S_\RR$},\\
      \bigl(\Omega_{v,a,j}\bigr)_{(a,j)\in [2d]\times[2d]} &\text{if $v \in S_\CC$}
    \end{cases}
\]
and $M$ is the matrix in $\GL_{nd}\bigl(\CC[G]\bigr)$ that represents the isomorphism~\eqref{period decomp} with respect to the bases $\mathcal{T}$ and $\mathcal{Y}$ so that
\[
M_{(a,j),(v,\alpha)} =
  \begin{cases}
     \sigma_v(z_a)\,\sum_{g \in G}\hat\sigma_v\bigl( g(Z)\bigr)g^{-1} &\text{if $v \in S_\RR$ and $\alpha = j$},\\
     \sigma_v(z_a)\,\sum_{g \in G}\hat\sigma_v\bigl( g(Z)\bigr)g^{-1} &\text{if $v\in S_\CC$ and $\alpha = j$},\\
     c\circ\sigma_v(z_a)\,\sum_{g \in G} c\circ\hat\sigma_v\bigl( g(Z)\bigr)g^{-1} &\text{if $v\in S_\CC$ and $\alpha = j+d$,}\\
     0 &\text{otherwise.}
   \end{cases}
\]
Writing $\bigl[M^{-1}\Phi\bigr]$ for the class of $M^{-1}\Phi$ in $K_1\bigl(\CC_p[G]\bigr)$ one therefore has
\begin{multline}\label{almost period comp}
 \chi_{G,p}\Bigl(H^{F/k}_{A,p}[0] \oplus J_{A,F_p}[-1],\,\mu_1\Bigr)
 = \partial_{G,p}\bigl([M^{-1}\Phi]\bigr) \\
 = \delta_{G,p}\bigl(\Nrd_{\CC_p[G]}(M^{-1}\Phi)\bigr)
 = -\delta_{G,p}\bigl(\Nrd_{\CC_p[G]}(M)\bigr) + \delta_{G,p}\bigl(\Nrd_{\CC_p[G]}(\Phi)\bigr).
\end{multline}
Now $M$ is equal to the product $M_1M_2$ for the block matrices $M_1 = \bigl(\sigma(z_a)I_d\bigr)_{\sigma\in \Sigma(k),a \in [n]}$ and $M_2 = \diag \bigl( (\sum_{g \in G}\hat\sigma( g(Z))g^{-1}) I_d\bigr)_{\sigma \in \Sigma(k)}$ and so $\Nrd_{\CC[G]}(M)$ is equal to the product
\begin{align*}
 \Nrd_{\CC[G]}(M_1)\Nrd_{\CC[G]}(M_2) &= \Nrd_{\CC[G]} \bigl( (\sigma(z_a))_{\sigma,a}\bigr)^d
      \Nrd_{\CC[G]}\biggl(\diag\Bigl(\sum_{g \in G}\hat\sigma\bigl( g(Z)\bigr)g^{-1}\Bigr)_{\sigma}\biggr)^d\\
&= \Nrd_{\CC[G]}\biggl(\Bigl(\sigma(z_a)\sum_{g \in G}\hat\sigma\bigl( g(Z)\bigr)g^{-1}\Bigr)_{a,\sigma}\biggr)^d.
\end{align*}
In addition, $\bigl(\sigma(z_a)\sum_{g \in G}\hat\sigma( g(Z))g^{-1}\bigr)_{a,\sigma}$ is the matrix of the natural isomorphism $\pi_F: \CC_p\otimes_{\QQ_p} F_p \cong \CC_p\otimes_{\ZZ_p}Y_{F,p}$ with respect to the $\ZZ_p[G]$-bases $\{Zz_a\}_{a \in [n]}$ and $\{\hat \sigma\}_{\sigma \in \Sigma(k)}$ of $\mathcal{O}_{F,p}$ and $Y_{F,p}$ and so
\begin{equation}\label{norm resolvents}
  \delta_{G,p}\bigl(\Nrd_{\CC_p[G]}(M)\bigr)
   = d \cdot\partial_{G,p}\biggl(\biggl[ \Bigl(\sigma(z_a)\sum_{g \in G}\hat\sigma\bigl( g(Z)\bigr)g^{-1}\Bigr)_{a,\sigma}\biggr]\biggr) = d\cdot\!\bigl(\mathcal{O}_{F,p},\pi_F,Y_{F,p}\bigr).
\end{equation}
Next we note that $\Nrd_{\CC_p[G]}(\Phi) = \prod_{v \in S_\infty}\Nrd_{\CC_p[G]}(\Phi_v)$ and that for each $v$ in $S_\infty$ also $\Nrd_{\CC_p[G]}(\Phi_v) = \sum_{\psi \in \Ir(G)} e_\psi\det_{\CC_p}\bigl(\Phi_v\bigm\vert V_\psi^{d[k_v:\RR]}\bigr)$ with
\begin{align*}
    \det\nolimits_{\CC_p}\Bigl(\Phi_v\Bigm\vert V_\psi^{d\,[k_v:\RR]}\Bigr) &=
     \begin{cases}
        \det\bigl((\Omega^+_{v,a,j})_{a,j}\bigr)^{\psi^+_v(1)}\cdot \det\bigl((\Omega^-_{v,a,j})_{a,j}\bigr)^{\psi^-_v(1)} &\text{if $v \in S_\RR$},\\
        \det\bigl((\Omega_{v,a,j})_{a,j}\bigr)^{\psi(1)} &\text{if $v \in S_\CC$.}
     \end{cases}
\end{align*}
Now if $v$ is real, then $\det(\Omega^+_{v,a,j})_{a,j} = \varpi_v^+\cdot \Omega_v^+(A)$ and $\det(\Omega^-_{v,a,j})_{a,j} = \varpi_v^-\cdot i^d\cdot \Omega_v^-(A)$ with $\varpi_v^+$ and $\varpi_v^-$ in $\{\pm 1\}$, whilst if $v$ is complex, then $\det(\Omega_{v,a,j})_{a,j} = \varpi_v\cdot i^d\cdot \Omega_v(A)$ with $\varpi_v$ in $\{\pm 1\}$. Thus, since each of the terms
\[
 \sum_{\psi \in \Ir_p(G)} e_\psi (\varpi_v^+)^{\psi^+_v(1)} \, (\varpi^-_{v})^{\psi^-_v(1)} = \Nrd_{\QQ_p[G]}\bigl ( \varpi_v^+(1+\tau_v)/2 + \varpi_v^-(1-\tau_v)/2\bigr)
\]
if $v$ is real, and $\sum_{\psi \in \Ir_p(G)} e_\psi \, \varpi_v^{\psi(1)} = \Nrd_{\QQ_p[G]} ( \varpi_v)$ if $v$ is complex, belong to the kernel of $\delta_{G,p}$, we conclude that $\delta_{G,p}\bigl(\Nrd_{\CC_p[G]}(\Phi)\bigr)$ is equal to
\begin{align*}
 &\delta_{G,p}\biggl(\sum_{\psi\in \Ir(G)}e_\psi \prod_{v \in S_\RR}\Bigl(\Omega^+_v(A)^{\psi^+_v(1)}\bigl(i^d\cdot \Omega^-_v(A)\bigr)^{\psi^-_v(1)}\Bigr) \prod_{v \in S_\CC}\bigl(i^d\cdot \Omega_v(A)\bigr)^{\psi(1)} \biggr)\\
= &\delta_{G,p}\biggl( \Bigl(\sum_{\psi\in \Ir(G)} w_\infty(\psi)^d\,e_\psi\Bigr)\Bigl(\sum_{\psi\in \Ir(G)}\Omega(A,\psi)e_\psi\Bigr)\biggr)\\
= &\delta_{G,p}\Bigl(w_\infty(F/k)^d\cdot \Omega(A,F/k)\Bigr).
\end{align*}
This formula combines with~\eqref{almost period comp} and~\eqref{norm resolvents} to give the claimed formula for the term $\chi_{G,p}\bigl(H^{F/k}_{A,p}[0] \oplus J_{A,F_p}[-1],\,\mu_1\bigr)$.
\end{proof}
This completes the proof of Theorem~\ref{comp prop}.

\section{Congruences between derivatives}\label{congruences section}

We assume in the sequel that, in addition to the hypotheses~\ref{hyp_a}--\ref{hyp_h}, the following standard conjecture is also valid.
\begin{conditions}\setcounter{condone}{8}
  \item\label{hyp_i} For each finite set of places $\Sigma$ of $k$ and each character $\psi$ in $\Ir(G)$ the $\Sigma$-truncated $\psi$-twisted Hasse-Weil $L$-function $L_\Sigma(A,\psi,s)$ of $A$ has an analytic continuation to $s=1$ where it has a zero of order $r_\psi := \dim_\CC \bigl(V_\psi\otimes_{\ZZ}A(F)\bigr)^G$.
\end{conditions}
Here $G$ acts diagonally on the tensor product and so $r_\psi$ is equal to the multiplicity with which $\psi$, and hence also $\check\psi$, appears in the representation $\CC\otimes_{\ZZ} A(F)$.

For each $\psi$ in $\Ir(G)$ we then write $\Lstar_\Sigma(A,\psi,1)$ for the the coefficient of $(s-1)^{r_\psi}$ in the Taylor expansion at $s=1$ of  $L_\Sigma(A,\psi,s)$.

In this section we use the computations of~\S\ref{bloch-kato} to give a reinterpretation, under the hypotheses~\ref{hyp_a}--\ref{hyp_i}, of the appropriate case of the
equivariant Tamagawa number conjecture of~\cite[Conjecture 4]{bufl99} as a family of congruence relations between the leading terms $\Lstar_{S_{\rm r}}(A,\psi,1)$ as
$\psi$ varies over $\Ir(G)$.
We next discuss several consequences of this reinterpretation and then, motivated by the results of \cite{Selmerstr}, we specialise to the case that the ($p$-completed)
Mordell-Weil group $A(F)_p$ is a projective $\ZZ_p[G]$-module. In this case we prove results that will subsequently enable us in~\S\ref{special cases} to give some
important new theoretical and numerical verifications of~\cite[Conjecture 4]{bufl99}.

\subsection{The general case}\label{generalcase}
 If $A$ and $F$ satisfy the hypotheses~\ref{hyp_g} and~\ref{hyp_i}, then~\cite[Conjecture 4]{bufl99} for the pair $\bigl(h^1(A_{/F})(1),\ZZ[G]\bigr)$ asserts the validity of an equality in $K_0\bigl(\ZZ[G],\RR[G]\bigr)$. In such a case we shall say that the `eTNC$_p$ for $\bigl(h^1(A_{/F})(1),\ZZ[G]\bigr)$ is valid' if for every isomorphism of fields $j:\CC\cong \CC_p$ the predicted equality is valid after projection under the homomorphism $j_{G,*}$ defined in~\eqref{induced rkt}.

 If $p$ does not divide $|G|$, then the algebra $\ZZ_p[G]$ is regular and it is straightforward to use the techniques described in~\cite[\S1.7]{bufl95} to give an explicit interpretation of these projections. If $p$ divides $|G|$, however, then obtaining an explicit interpretation is in general very difficult (see, for example, the efforts made by Bley in~\cite{Bley, Bley2}).

 The following result is thus of some interest since, as we shall see later, it can be combined with the structure results obtained in \cite{Selmerstr} to show that under certain natural conditions one can explicitly interpret, and verify, the eTNC$_p$ for $\bigl(h^1(A_{/F})(1),\ZZ[G]\bigr)$ even if $G$ is non-abelian, $p$ divides $|G|$ and the rank of $A(F)$ is strictly positive.

For each non-archimedean place $v$ of $k$, we decompose the non-ramified characteristic $u_v$ as $\sum_{\psi\in \Ir(G)}u_{v,\psi}\cdot e_\psi$  with each  $u_{v,\psi}$ in $\CC^\times$. For each $\psi$ in $\Ir(G)$ we then define a modified global Galois-Gauss sum by setting
\begin{equation}\label{ggs def}
  \tau^*\bigl(\QQ,\Ind_k^\QQ\psi\bigr) := u_\psi \cdot \tau\bigl(\QQ,\Ind_k^\QQ\psi\bigr).
\end{equation}
with $u_\psi := \prod_{v \in S_{\rm r}^k} u_{v,\psi}$.

\begin{theorem}\label{equivalence theorem}
 Assume that $A$ and $F$ satisfy the hypotheses~\ref{hyp_a}--\ref{hyp_i}.
Then the element
\begin{equation*}
 \calLstar_{A,F/k} := \sum_{\psi \in \Ir(G)} e_\psi\, \frac{\Lstar_{S_{\rm r}}(A,\check{\psi},1)\cdot \tau^*\bigl(\QQ,\Ind_k^\QQ\psi\bigr)^{d}} {\Omega(A,\psi)\cdot w_\infty(\psi)^{d}}
\end{equation*}
belongs to $\zeta\bigl(\CC[G]\bigr)^\times$ and the eTNC$_p$ for $\bigl(h^1(A_{/F})(1),\ZZ[G]\bigr)$ is valid if and only if for every isomorphism of fields $j: \CC\cong \CC_p$ one has
\begin{equation}\label{etnc equiv}
  \delta_{G,p}\bigl(j_{*}(\calLstar_{A,F/k})\bigr) = \chi_j(A,F/k)
\end{equation}
\end{theorem}
\begin{proof} Since $\Lstar_{S_{\rm r}}(A,\check{\psi},1)\neq 0$ for all $\psi$ in $\Ir(G)$ the element $\calLstar_{A,F/k}$ clearly belongs to $\zeta(\CC[G])^\times$.

We next recall that~\cite[Conjecture 4(iv)]{bufl99} for the pair $\bigl(h^1(A_{/F})(1), \ZZ[G]\bigr)$ asserts
\[R\Omega\bigl(h^1(A_{/F})(1), \ZZ[G]\bigr) = -\delta_G\bigl(\Lstar(_{\QQ[G]}h^1(A_{/F})(1),0)\bigr)\]
where $\delta_G$ is the `extended boundary homomorphism' $\zeta \bigl( \RR[G]\bigr)^\times \to K_0\bigl(\ZZ[G],\RR[G]\bigr)$ defined in~\cite[Lemma 9]{bufl99} and $\Lstar\bigl(_{\QQ[G]}h^1(A_{/F})(1),0\bigr)$ is the leading term at $s=0$ of the $\zeta\bigl(\CC[G]\bigr)$-valued $L$-function $L\bigl(_{\QQ[G]}h^1(A_{/F})(1),s\bigr)$ defined in
\cite[\S4.1]{bufl99}.

Noting that for each isomorphism $j:\CC\cong \CC_p$ one has $j_{G,*}\circ \delta_G = \delta_{G,p}\circ j_*$, and recalling the result of Proposition~\ref{comparison lemma}, it follows that the eTNC$_p$ for $\bigl(h^1(A_{/F})(1),\ZZ[G]\bigr)$ is valid if and only if for every isomorphism $j$ one has
\begin{multline*} \delta_{G,p}\Bigl(j_*\bigl(\Lstar(_{\QQ[G]}h^1(A_{/F})(1),0)\bigr)\Bigr)\\ = \chi_j(A,F/k) + \chi_j^{\rm loc}(A,F/k) - \sum_{v \in S\cup S_p}\delta_{G,p}\bigl(L_v(A,F/k)\bigr).\end{multline*}
Now hypotheses~\ref{hyp_e} and~\ref{hyp_h} combine to imply that $S_{\rm r} = (S\cup S_p)\setminus (S_{\rm b}\cup S_p)$ and for each place $v$ in this set the term $L_v(A,F/k)$ is equal to the value at $s=0$ of the Euler factor at $v$ that occurs in the definition of $L\bigl(_{\QQ[G]}h^1(A_{/F})(1),s\bigr)$ (by~\cite[Remark 7]{bufl99}). In view of the formula for $\chi_j^{\rm loc}(A,F/k)$ that is given in Theorem~\ref{comp prop}, one therefore finds that the above equality is valid if and only if
\[ \delta_{G,p}\biggl(j_*\Bigl(\frac{\Lstar_{S_{\rm r}}\bigl(_{\QQ[G]}h^1(A_{/F})(1),0\bigr)\cdot\tau^*(F/k)^{d}}{\Omega(A,F/k)\cdot w_\infty(F/k)^{d}} \Bigr)\biggr) = \chi_j(A,F/k).\]
It now suffices to show that the quotient on the left hand side of this equality is equal to $\calLstar_{A,F/k}$ and this follows from a straightforward comparison of all of the terms involved and then noting that the definition of the truncated $L$-function $L_{S_{\rm r}}\bigl(_{\QQ[G]}h^1(A_{/F})(1),s\bigr)$ implies $\Lstar_{S_{\rm r}}\bigl(_{\QQ[G]}h^1(A_{/F})(1),0\bigr) = \sum_{\psi \in \Ir(G)}e_\psi \Lstar_{S_{\rm r}}(A,\check{\psi},1)$.
\end{proof}

If it is valid, then the equality~\eqref{etnc equiv} can be combined with the theory of organising matrices developed in~\cite{omac} to extract from the  element
$\calLstar_{A,F/k}$ a range of detailed information about the arithmetic of $A$ over $F$. To give an explicit example of such an implication we assume, motivated by the
results of Theorem 2.7 and Corollary 2.10 in \cite{Selmerstr}, that there exists a surjective homomorphism of $\ZZ_p[G]$-modules of the form
\[ \pi: \Sel_p(A_{F})^\vee \to \Pi = \Pi^{\rm pr} \oplus \Pi^{\rm npr}\]
where $\Pi$ is a trivial source $\ZZ_p[G]$-module, $\Pi^{\rm pr}$ is a projective $\ZZ_p[G]$-module and $\pi$ induces, upon passage to $G$-coinvariants, an isomorphism
$A(k)^*_p \cong \Pi_G$ (via the relevant canonical short exact sequence of the form (\ref{sha-selmer})). The notion of `trivial source $\ZZ_p[G]$-module'
that we use here corresponds to the one defined in \cite[\S 2.3.2]{Selmerstr}. For the purpose of applying the result \cite[Corollary 2.13]{omac} to prove
Proposition \ref{consequences} below we however warn the reader that this terminology differs from the one introduced in \S 2.2 in loc. cit., where trivial source modules
are instead called `permutation modules'.
In this case we also write $\Upsilon_{\rm pr}$ for the subset of $\Ir(G)$ comprising characters for which the homomorphism
$e_\psi(\CC\otimes_\ZZ \pi)$ is bijective and then define an idempotent in $\zeta\bigl(\QQ[G]\bigr)$ by setting $e_{\rm pr} := \sum_{\psi \in \Upsilon_{\rm pr}} e_\psi$.

Given a natural number $m$ and a matrix $M$ in $\operatorname{M}_m\bigl(\ZZ_p[G]\bigr)$ we write $M'$ for the corresponding `generalised adjoint matrix' in
$\operatorname{M}_m\bigl(\QQ_p[G]\bigr)$ that is defined by Johnston and Nickel in \cite[\S 3.6]{jn}.
We then write $\mathcal{A}_p(G)$ for the $\zeta\bigl(\ZZ_p[G]\bigr)$-submodule of
$\zeta\bigl(\QQ_p[G]\bigr)$ given by the set
\begin{equation*}
  \Bigl\{ x\in \zeta\bigl(\QQ_p[G]\bigr): \text{ if } m > 0 \text{ and }  M \in \operatorname{M}_m\bigl(\ZZ_p[G]\bigr) \text{ then }xM' \in \operatorname{M}_m\bigl(\ZZ_p[G]\bigr)\Bigr\}.
\end{equation*}
For more details about this module see Remark~\ref{denom ideal} below.

\begin{proposition}\label{consequences}
Assume that $A$ and $F$ satisfy the hypotheses~\ref{hyp_a}--\ref{hyp_i} and that the equality~\eqref{etnc equiv} is valid.
Fix a surjective homomorphism of $\ZZ_p[G]$-modules
\[ \pi: \Sel_p(A_{F})^\vee \to \Pi\]
as above and an element
 $a$ of $\mathcal{A}_p(G)$. For each homomorphism $\theta: A^t(F)_p \to A(F)_p^*$ of $\ZZ_p[G]$-modules and each isomorphism of fields $j:\CC\cong \CC_p$ set
\[ R_j(\theta) := \Nrd_{\CC_p[G]} \bigl( (\CC_p\otimes_{\ZZ_p}\theta)\circ (\lambda_{A,F}^{{\rm NT},j})^{-1}\bigr ).\]

Then for each such $\theta$ and $j$ the element $ a\, R_j(\theta)\,j_*(\calLstar_{A,F/k})\, e_{\rm pr}$ belongs to $\ZZ_p[G]$ and annihilates both
modules $\sha_p(A^t_{F})$ and $\Pi^{\rm npr}$. In particular, the element
\begin{equation*}
 \mathcal{L}_{A,F/k} := \sum_{\psi \in \Ir(G)}e_\psi\, \frac{L_{S_{\rm r}}(A,\check{\psi},1)\cdot \tau^*\bigl(\QQ,\Ind_k^\QQ\psi\bigr)^{d} }{\Omega(A,\psi)\cdot w_\infty(\psi)^{d}}
\end{equation*}
is such that $a\cdot j_*(\mathcal{L}_{A,F/k})$ belongs to $\ZZ_p[G]$ and annihilates $\sha_p(A^t_{F})$.
\end{proposition}
\begin{proof} Hypothesis~\ref{hyp_a} implies that the module $A^t(F)_p$ is $\ZZ_p$-free. Lemma~\ref{canisos} therefore implies, in the terminology of~\cite{omac}, that $R\Gamma_f\bigl(k,T_{p,F}(A)\bigr)$ is an admissible complex of $\ZZ_p[G]$-modules and equation~\eqref{etnc equiv} asserts that $j_*(\calLstar_{A,F/k})$ is a characteristic element for the pair $\bigl(R\Gamma_f(k,T_{p,F}(A)),\,(\lambda^{{\rm NT},j}_{A,F})^{-1}\bigr)$. The first claim is therefore a direct consequence of~\cite[Corollary 2.13]{omac} and the isomorphisms
\[ \sha_p(A^t_{F}) \cong \sha_p(A_{F})^\vee \cong \bigl(\Sel_p(A_{F})^\vee\bigr)_{\rm tor} \cong H^2\bigl(R\Gamma_f(k,T_{p,F}(A))\bigr)_{\rm tor},\]
where the first isomorphism is induced by the Cassels-Tate pairing, the second follows from the short exact sequence (\ref{sha-selmer}) (with $L=F$) and the last is a consequence of Lemma~\ref{canisos}.

 To deduce the second claim we note that, under our hypotheses~\ref{hyp_a}--\ref{hyp_i}, the natural projection map
\[ \Sel_p(A_{F})^\vee \to \bigl(\Sel_p(A_{F})^\vee\bigr)_G \cong \Sel_p(A_{k})^\vee \to A(k)_p^*=:\Pi\]
is a homomorphism $\pi$ of the required type. Further, in this case one has $\Pi^{\rm pr} = 0$ and so hypothesis~\ref{hyp_i} implies $\Upsilon_{\rm pr}$ is equal to $\bigl\{\psi \in \Ir(G): L_{S_{\rm r}}(A,\check{\psi},1) \not= 0\bigr\}$ and hence that $\mathcal{L}_{A,F/k} = \mathcal{L}_{A,F/k}\,e_{\rm pr} = \calLstar_{A,F/k}\,e_{\rm pr}$. It therefore suffices to show that $R_j(\theta)e_{\rm pr} = e_{\rm pr}$ and this is true because in this case the space $e_{\rm pr}\Hom_\CC\bigl(\CC\cdot A(F),\CC\bigr)$ vanishes.
\end{proof}

\begin{remark}\label{explicit cong}
 The conjectural equality~\eqref{etnc equiv} implies, via Proposition~\ref{consequences}, a family of explicit congruence relations between the complex numbers $A(\psi)$ that are defined by the equalities $A(\psi)\,e_\psi = a\,R_j(\theta)\,j_*(\calLstar_{A,F/k})\,e_\psi $ as $\psi$ varies over $\Upsilon_{\rm pr}$. This is because
$$
 a\,R_j(\theta)\,j_*(\calLstar_{A,F/k})\,e_{\rm pr} = \sum_{\psi \in \Upsilon_{\rm pr}} A(\psi)\,e_\psi =  \,|G|^{-1} \sum_{g \in G}g\sum_{\psi \in \Upsilon_{\rm pr}} \psi(1)\,\check\psi(g)\,A(\psi)
$$
and this sum belongs to $\zeta\bigl(\ZZ_p[G]\bigr)$ if and only if the elements $A(\psi)$ satisfy all of the following conditions:

 \begin{itemize}
 \item[(i)] $A(\psi)\in \ZZ_p[\psi]$ for all $\psi \in \Upsilon_{\rm pr}$;
 \item[(ii)] $\alpha(A(\psi)) = A(\alpha\circ\psi)$ for all $\psi \in \Upsilon_{\rm pr}$ and $\alpha\in G_{\QQ_p(\psi)/\QQ_p}$;
 \item[(iii)] $\sum_{\psi \in \Upsilon_{\rm pr}}\psi(1)\check\psi(g)A(\psi)\equiv 0 \pmod{|G|\ZZ_p}$ for all $g \in G$.
\end{itemize}
\end{remark}

\begin{remark}\label{denom ideal} Ideals of the form $\mathcal{A}_p(G)$ were introduced by Nickel in~\cite{nickel} and have been computed extensively by Johnston and Nickel in~\cite{jn}.
For example, if $M= I_m$, then $M' = I_m$ so that $\mathcal{A}_p(G) \subseteq \zeta\bigl(\ZZ_p[G]\bigr)$ and it is shown in loc. cit. that this inclusion is an equality if and only if the order of the commutator subgroup of $G$ is not divisible by $p$. More generally, for each $M$ in $\operatorname{M}_d\bigl(\ZZ_p[G]\bigr)$ the matrix
$M'$ belongs to $\operatorname{M}_m(\mathcal{M})$ for any maximal order $\mathcal{M}$ in $\QQ_p[G]$ that contains $\ZZ_p[G]$ (cf.~\cite[Lemma. 4.1]{nickel}) and so Jacobinski's description in~\cite{Jac} of the central conductor of
$\mathcal{M}$ in $\ZZ_p[G]$ implies, for example, that for any $\QQ_p^c$-valued character $\psi$ of $G$ the element
$\psi(1)^{-1}|G|e_\psi$ belongs to $\ZZ_p[\psi]\otimes_{\ZZ_p}\mathcal{A}_p(G)$ where $\ZZ_p[\psi]$ is the subring of $\QQ_p^c$ that is generated over $\ZZ_p$ by the values of $\psi$. This gives an easy `lower bound' on $\mathcal{A}_p(G)$ (but which is, in most cases, not best possible).\end{remark}

\begin{remark} For the explicit computation of terms of the form $\chi_j(A,F/k)$ and $R_j(\theta)$ occurring in Theorem \ref{equivalence theorem} and Proposition
\ref{consequences} respectively in the case that $G$ is cyclic of $p$-power order we refer the reader to \cite{bleymc} where, in addition, Theorem \ref{equivalence theorem}
is used in order to discuss further explicit
(conjectural) properties of elements of the form $a\, R_j(\theta)\,j_*(\calLstar_{A,F/k})\, e_{\rm pr}$ in such settings.
\end{remark}

\subsection{The case of projective Mordell-Weil groups}\label{projective case}

In the rest of this article we consider the conjectural equality~\eqref{etnc equiv} in the case that $A(F)_p$ and $A^t(F)_p$ are both projective $\ZZ_p[G]$-modules. (In this regard, note that if $A$ is principally polarised, then the $\ZZ_p[G]$-modules $A(F)_p$ and $A^t(F)_p$ are isomorphic.)
This corresponds to taking the module $\Pi$ in Proposition~\ref{consequences} to be equal to $\Pi^{\rm pr} = A(F)_p^*$ and the more general case that $\Pi = A(F)_p^*$ is a trivial source $\ZZ_p[G]$-module will be considered in a future article.

In the following result we use the notion of non-commutative Fitting invariant that was introduced by Parker~\cite{parker} and studied further by Nickel~\cite{nickel}.

\begin{proposition}\label{etnc proj case}
Assume that $A$ and $F$ satisfy~\ref{hyp_a}--\ref{hyp_i}. Assume also that $A(F)_p$ and $A^t(F)_p$ are both projective $\ZZ_p[G]$-modules.

Then the projective dimension of the $\ZZ_p[G]$-module $\sha_p(A^t_{F})$ is at most one and there exists an isomorphism $\iota_p: A(F)^*_p \cong A^t(F)_p$ of $\ZZ_p[G]$-modules.

For each $\psi\in \Ir(G)$ now set
\[
 R_{\psi,\iota_p}^{{\rm NT},j}(A_{/F}) := \det {}_{\CC_p}\Bigl( (\CC_p\otimes_{\ZZ_p}\iota_p)\circ\lambda^{{\rm NT},j}_{A,F}\Bigm\vert \bigl(V_{j\circ\psi}\otimes_{\ZZ}A^t(F)\bigr)^G\Bigr) \in \CC_p^{\times}.
\]

Then the equality~\eqref{etnc equiv} is valid if and only if the (non-commutative) Fitting invariant of the $\ZZ_p[G]$-module $\sha_p(A^t_{F})$ is generated by the set of elements of the form $u\cdot\calLstar_{A,F/k,j,\iota_p}$ where $u$ is in $\Nrd_{\QQ_p[G]}(K_1(\ZZ_p[G]))$ and
\begin{equation*}
 \calLstar_{A,F/k,j,\iota_p} := \sum_{\psi \in \Ir(G)} j_*\biggl(e_\psi \cdot \frac{\Lstar_{S_{\rm r}}(A,\check{\psi},1)\cdot \tau^*\bigl(\QQ,\Ind_k^\QQ\psi\bigr)^{d}}{\Omega(A,\psi)\cdot w_\infty(\psi)^{d}}\biggr)  \frac{1}{R_{\psi,\iota_p}^{{\rm NT},j}(A_{/F})}.
\end{equation*}
\end{proposition}
\begin{proof}
The assumed projectivity of $A(F)_p$ implies that its $\ZZ_p$-linear dual $A(F)_p^*$ is also a projective $\ZZ_p[G]$-module. The existence of an isomorphism $\iota_p$ therefore follows from Swan's Theorem~\cite[Theorem 32.1]{curtisr} and the fact that $\lambda^{{\rm NT},j}_{A,F}$ induces an isomorphism of $\CC_p[G]$-modules $\CC_p\cdot A^t(F)_p \cong \Hom_{\CC_p}\bigl(\CC_p\otimes_\ZZ A(F),\CC_p\bigr) = \CC_p\cdot A(F)^*_p$.

We now write $C^\bullet$ for the complex $A^t(F)_p \xrightarrow{0} A(F)^*_p$, where the first term occurs in degree one. Then, since $A(F)_p^*$ is a projective $\ZZ_p[G]$-module, we may choose a $\ZZ_p[G]$-equivariant section to the natural surjection
\[ H^2\bigl(R\Gamma_f(k,T_{p,F}(A))\bigr) \cong \Sel_p(A_{F})^\vee \to A(F)^*_p = H^2(C^\bullet).\]
Since the kernel of this surjection is isomorphic (via the Cassels-Tate pairing) to $\sha_p(A^t_{F})$, such a section induces a short exact sequence of $\ZZ_p[G]$-modules
\[ 0\to H^2(C^\bullet) \to H^2\bigl(R\Gamma_f(k,T_{p,F}(A))\bigr) \to \sha_p(A^t_{F})\to 0\]
and hence also an exact triangle in $D^{\rm p}\bigl(\ZZ_p[G]\bigr)$ of the form
\[ C^\bullet \to R\Gamma_f(k,T_{p,F}(A)) \to \sha_p(A^t_{F})[-2] \to C^\bullet [1].\]
In particular, since $\sha_p(A^t_{F})[-2]$ belongs to $D^{\rm p}\bigl(\ZZ_p[G]\bigr)$, the projective dimension of the finite $\ZZ_p[G]$-module $\sha_p(A^t_{F})$ is finite, and hence at most one (by~\cite[Chapter VI, (8.12)]{brown}), as claimed.

Further, if we use the cohomology sequence of the above triangle to identify $\CC_p\cdot H^i(C^\bullet)$ and $\CC_p\cdot H^i\bigl(R\Gamma_f(k,T_{p,F}(A))[1]\bigr)$ in all degrees $i$, then this triangle combines with the additivity criterion of~\cite[Corollary 6.6]{additivity} to imply that the element $\chi_j(A,F/k) = - \chi_{G,p}\bigl(R\Gamma_f(k,T_{p,F}(A)),\lambda^{{\rm NT},j}_{A,F}\bigr)$ is equal to
\begin{align*}
 -\chi_{G,p}\bigl(& C^\bullet,\lambda^{{\rm NT},j}_{A,F}\bigr) - \chi_{G,p}\bigl(\sha_p(A^t_{F})[-2],0\bigr)\\
    & = -\chi_{G,p}\bigl(C^\bullet,\lambda^{{\rm NT},j}_{A,F}\bigr) + \chi_{G,p}\bigl(C^\bullet,(\CC_p\otimes_{\ZZ_p}\iota_p)^{-1}\bigr) - \chi_{G,p}\bigl(\sha_p(A^t_{F})[-2],0\bigr)\\
   & = \delta_{G,p}\bigl(\Nrd_{\CC_p[G]}((\CC_p\otimes_{\ZZ_p}\iota_p)\circ\lambda^{{\rm NT},j}_{A,F})\bigr) - \chi_{G,p}\bigl(\sha_p(A^t_{F})[-2],0\bigr).
\end{align*}
Here the first displayed equality is valid as $\chi_{G,p}(C^\bullet,(\CC_p\otimes_{\ZZ_p}\iota_p)^{-1}) = 0$ and the second because the difference  $-\chi_{G,p}(C^\bullet,\lambda^{{\rm NT},j}_{A,F}) + \chi_{G,p}(C^\bullet,(\CC_p\otimes_{\ZZ_p}\iota_p)^{-1})$ is represented by the triple $\bigl(A^t(F)_p, (\CC_p\otimes_{\ZZ_p}\iota_p)\circ\lambda^{{\rm NT},j}_{A,F},A^t(F)_p\bigr)$.

The equality~\eqref{etnc equiv} is therefore valid if and only if one has
\begin{align*}
\delta_{G,p}\Bigl(j_*(\calLstar_{A,F/k})\cdot & \Nrd_{\CC_p[G]}\bigl((\CC_p\otimes_{\ZZ_p}\iota_p)\circ\lambda^{{\rm NT},j}_{A,F}\bigr)^{-1}\Bigr)\\
  & =  \delta_{G,p}\bigl(j_*(\calLstar_{A,F/k})\bigr) - \delta_{G,p}\Bigl(\Nrd_{\CC_p[G]}\bigl((\CC_p\otimes_{\ZZ_p}\iota_p)\circ\lambda^{{\rm NT},j}_{A,F}\bigr)\Bigr)\\
    & =  \chi_j(A,F/k) - \delta_{G,p}\Bigl(\Nrd_{\CC_p[G]}\bigl((\CC_p\otimes_{\ZZ_p}\iota_p)\circ\lambda^{{\rm NT},j}_{A,F}\bigr)\Bigr)  \\
   & = -\chi_{G,p}\bigl(\sha_p(A^t_{F})[-2],0\bigr).
\end{align*}

In addition, one has $\Nrd_{\CC_p[G]}\bigl((\CC_p\otimes_{\ZZ_p}\iota_p)\circ\lambda^{{\rm NT},j}_{A,F}\bigr)e_{j\circ\psi} = R_{\psi,\iota_p}
^{{\rm NT},j}(A_{/F})e_{j\circ\psi}$ for each $\psi$ in $\Ir(G)$ and so
\[ j_*(\calLstar_{A,F/k})\cdot\Nrd_{\CC_p[G]}\bigl((\CC_p\otimes_{\ZZ_p}\iota_p)\circ\lambda^{{\rm NT},j}_{A,F}\bigr)^{-1} = \calLstar_{A,F/k,j,\iota_p}.\]
Given this, a straightforward exercise (comparing the explicit definitions of refined Euler characteristic and non-commutative Fitting invariants) shows that the above formula for $\chi_{G,p}\bigl(\sha_p(A^t_{F})[-2],0\bigr)$ is valid if and only if $\Fit_{\ZZ_p[G]}\bigl(\sha_p(A^t_{F})\bigr)$ is generated by the set of elements $u\cdot\calLstar_{A,F/k,j,\iota_p}$ with $u$ in $\Nrd_{\QQ_p[G]}(K_1(\ZZ_p[G]))$, as claimed. \end{proof}

\subsection{Dihedral congruences for elliptic curves}\label{dt ec}

We now investigate the criterion of Proposition~\ref{etnc proj case} in the case that $A$ is an elliptic curve (so $A = A^t$) and $F/k$ is dihedral (in the sense of
Mazur and Rubin~\cite{mr2}).

Thus, as before,
we have an odd prime $p$ and a Galois extension $F/k$ of group $G$ with $p$-Sylow subgroup $P$ and we assume that $P$ is an abelian (normal) subgroup of $G$ of index
two and that the conjugation action of any lift to $G$ of the generator of $G/P$ inverts elements of $P$.
In particular, the degree of $F/k$ is equal to $2p^n$ for some $n \geq 1$ and $K/k$ is a quadratic extension. We fix an element $\tau$ of order $2$ in $G$. We set
$\eins := \eins_G$ and write $\epsilon$ for the unique non-trivial linear character of $G$.

In the following result, we will be interested in the case when $A$ has rank one over $K$. In this case we write $\rho_{A}$ for the unique linear character which does
\emph{not} occur in the $\CC[G]$-module $\CC\otimes_\ZZ A(K)$. Hence $\rho_A = \eins$ if $\rk(A_k) = 0$ and $\rho_A=\epsilon$ otherwise.

We also set $\ZZ' := \ZZ\bigl[\tfrac{1}{2}\bigr]$.

\begin{proposition}\label{Q_prop}
  Assume that the elliptic curve $A$, odd prime $p$ and dihedral extension $F/k$ satisfy the hypotheses~\ref{hyp_a}--\ref{hyp_e} and~\ref{hyp_g}, that all places above $p$ split in $K/k$, that $\rk(A_K) = 1$ and that $\sha_p(A_{K}) =  0$.

  Then $\sha_p(A_{F}) =  0 $ and there is a point $Q$ in $A(F)$ with $\tau(Q) = -(-1)^{\rk(A_k)}\,Q$ which generates a $\ZZ'[G]$-module that is isomorphic to $\ZZ'[G]\bigl(1-(-1)^{\rk(A_k)}\,\tau\bigr)$ and has finite, prime-to-$p$, index in $\ZZ'\otimes_{\ZZ}A(F)$.

  In particular, one has $r_{\rho_A} = 0$ and $r_\psi = 1$ for all $\psi \in \Ir(G)\setminus\{\rho_A\}$.
\end{proposition}

\begin{proof}
  Under the stated hypotheses, \cite[Corollary 2.10(ii)]{Selmerstr} implies that $\sha_p(A_{F})$ vanishes, that $\Sel_p(A_{F})^\vee$ is a projective $\ZZ_p[G]$-module
  and that the multiplicity with which each $\rho$ in $\Ir(P)$ occurs in the representation $\CC_p\cdot\Sel_p(A_{F})^\vee$ is equal to one.

  Roiter's Lemma (cf.~\cite[(31.6)]{curtisr}) therefore implies (via the exact sequence (\ref{sha-selmer}))
  the existence of an exact sequence of $\ZZ'[G]$-modules
   \[
  \xymatrix@1@C+3ex{ 0\ar[r]& \ZZ'[G](1-(-1)^{\rk(A_k)}\tau)\ar[r] &  \ZZ'\otimes_{\ZZ} A(F)_{\rm tf}\ar[r] & X\ar[r]& 0,}
  \]
  where the group $X$ is both finite and of order prime to $p$.

  Since the group $A(F)_{\rm tor}$ is also finite of order prime to $p$ (by hypothesis~\ref{hyp_a}) it follows that any point $Q$ of
  $A(F)$ whose projection in $A(F)_{\rm tf}$ is equal to the image of $1-(-1)^{\rk(A_k)}\tau$ multiplied by a large enough power of 2 has the properties described above.

The above description of $\ZZ'\otimes_{\ZZ}A(F)$ also implies that the $\CC[G]$-module $\CC\otimes_{\ZZ}A(F)$ is isomorphic to $\CC[G]\bigl(1-(-1)^{\rk(A_k)}\tau\bigr)$
and the claimed formulas for $r_\psi$ are then easily verified by explicit computation.
\end{proof}

For each subgroup $H$ of $G$ and character $\rho$ in $\Ir(H)$ we set
\[ T_\rho:= \sum_{h \in H}\rho(h^{-1})h \in \zeta\bigl(\CC[H]\bigr).\]
For any point $R$ of $A(F)$ and any $\psi \in \Ir(G)$ we then define a non-zero complex number
 \begin{equation*}
  h_{F,\psi}(R) := \frac{\psi(1)}{2|G|} \cdot\Bigl\langle T_\psi (R), \, T_{\check\psi}(R) \Bigr\rangle_F
 \end{equation*}
 where $\langle \cdot,\cdot\rangle_F$ is the $\CC$-linear extension of the N\'eron-Tate height on $A$, defined relative to the field $F$.

In the next result we assume the hypotheses of Proposition~\ref{Q_prop} to be satisfied and fix a point $Q$ as in that result. For each $\psi\in\Ir(G)$ we then obtain a non-zero complex number $\mathcal{Q}_\psi = \mathcal{Q}_{Q,\psi}$ by setting
   \begin{equation}\label{Qdef_eq}
     \mathcal{Q}_\psi = u_{\psi}\cdot \frac{ \sqrt{ d_\psi} \cdot \Lstar_{S_{\rm r}} (A,\psi,1)}{\Omega(A,\psi)\cdot H_{\psi}}
\end{equation}
where the quantity $u_\psi$ is as in~\eqref{ggs def}. Further we define
\[
 d_{\eins} =  \vert d_k\vert,\quad
 d_{\epsilon} = \vert d_K/d_k\vert, \quad\text{ and }\quad
 d_{\Ind^G_P(\chi)} = \vert d_{K}\vert\cdot N\!f(\chi)
\]
for all $\chi\in \Ir(P)\setminus\{\eins_P\}$ and
\[ H_{\rho_A} = 1 \qquad \text{ and }\qquad H_\psi = h_{F,\psi}(Q), \qquad \text{for all $\psi\neq \rho_A$.}
\]
Here $d_E$ denotes the absolute discriminant of a number field $E$ and for any finite dimensional complex character $\psi$ of $G_E$ we write $N\!f(\psi)$ for the absolute norm of its Artin conductor.

We also note that hypothesis~\ref{hyp_i} combines with Proposition~\ref{Q_prop} to imply that the leading term  $L^\star_{S_{\rm r}}(A,\psi,1)$ is equal to the value $L_{S_{\rm r}}(A,\psi,1)$ for $\psi = \rho_A$ and to the first derivative $L'_{S_{\rm r}}(A,\psi,1)$ for  $\psi$ in $\Ir(G)\setminus \{\rho_A\}$.
\begin{theorem}\label{exp cong}
 Fix an odd prime $p$, a dihedral extension $F/k$ of degree $2p^n$ and an elliptic curve $A$ over $k$. Assume that $A$ and $F$ satisfy the hypotheses~\ref{hyp_a}--\ref{hyp_h}, that all places above $p$ split in $K/k$, that $\rk(A_K)=1$ and that $\sha_p(A_{K}) =  0$. Fix a point $Q$ in $A(F)$ as given by Proposition~\ref{Q_prop}.

 Then the equality~\eqref{etnc equiv} is valid if and only if the following conditions are satisfied.
   \begin{enumerate}
     \item[(i)]\label{exp cong_i}
       For each $\psi$ in $\Ir(G)$ the number $\mathcal{Q}_\psi$ defined above belongs to $\QQ(\psi)$, is a unit at all primes above $p$ and satisfies $(\mathcal{Q}_\psi)^\alpha = \mathcal{Q}_{\psi^\alpha}$ for all $\alpha$ in $\Gal\bigl({\QQ(\psi)/\QQ}\bigr)$.
     \item[(ii)]\label{exp cong_ii}
       For all $\pi \in P$ one has a congruence
       \begin{equation}\label{first congs}
          \mathcal{Q}_{\eins_G}\cdot \mathcal{Q}_{\epsilon} \equiv
           -\sum_{\chi \in \Ir(P)\setminus\{\eins_P\}} \check\chi(\pi)\, \mathcal{Q}_{\Ind^G_P(\chi)} \pmod{p^n\ZZ_{(p)}}
       \end{equation}
     where $\ZZ_{(p)}$ denotes the localisation of $\ZZ$ at $p$.
   \end{enumerate}
\end{theorem}

\begin{remark}\label{gen exp cong}
 If $Q'$ is any element of $A(F)$ with $T_\psi(Q') \neq 0$ for all $\psi \in \Ir(G)\setminus \{\rho_A\}$, then for each $\psi \in \Ir(G)$ one can define a non-zero complex number $\mathcal{Q}'_\psi = \mathcal{Q}_{Q',\psi}$ just as in~\eqref{Qdef_eq}.
Our proof of Theorem~\ref{exp cong} will actually show that if the conditions (i) and (ii) are valid with each $\mathcal{Q}_\psi$ replaced by $\mathcal{Q}'_\psi$, then $Q'$ is a $\ZZ_p[G]$-generator of $A(F)_p$ and the equality~\eqref{etnc equiv} is valid.
\end{remark}

\begin{remark}\label{n1_rem}
  If $n=1$, so the extension $F/k$ in Theorem~\ref{exp cong} is dihedral of degree $2p$, then each term $\check\chi(\pi)$ is congruent to $1$ modulo the unique prime above $p$ in $\QQ(\chi)$ and so the congruences~\eqref{first congs} reduce to a single congruence
  \begin{equation}\label{prime_cong_eq}
    \mathcal{Q}_{\eins_G}\cdot \mathcal{Q}_\epsilon \equiv
   -2 \sum_{\substack{\psi \in \Ir(G)\\ \dim(\psi) = 2 }} \mathcal{Q}_\psi\pmod{p\ZZ_{(p)}}.
  \end{equation}
\end{remark}

\begin{proof}[Proof of Theorem~\ref{exp cong}]
 It suffices to prove that the criterion of Proposition~\ref{etnc proj case} is valid if and only if both the condition (i) and the congruences
 in~\eqref{first congs} are valid. Now, since $\sha_p(A_F) =  0$ (by \cite[Corollary 2.10(ii)]{Selmerstr}), the criterion of Proposition~\ref{etnc proj case} is equivalent to asking that the element
 $\calLstar_{A,F/k,j,\iota_p}$ belongs to  $\ker(\delta_{G,p})$.
In view of Lemma~\ref{last one} below it is thus enough to show that the condition (i) and the congruences in~\eqref{first congs} are valid if and only
if the element $\mathcal{L}' := j_*(\sum_{\psi \in \Ir(G)} \mathcal{Q}_\psi e_\psi)$ belongs to  $\ker(\delta_{G,p})$.

To investigate the element $\delta_{G,p}(\mathcal{L}')$ we fix a maximal $\ZZ_p$-order $\mathfrak{M}_p$ in $\QQ_p[G]$ that contains $\ZZ_p[G]$. Then  $\delta_{G,p}(\mathcal{L}')$
belongs to the subgroup $K_0\bigl(\ZZ_p[G],\QQ_p[G]\bigr)_{\rm tor}$ of $K_0\bigl(\ZZ_p[G],\CC_p[G]\bigr)$ if and only if
$\mathcal{L}'\in \Nrd_{\QQ_p[G]}\bigl(\mathfrak{M}_p^\times\bigr)$ and this condition is satisfied if and only if the conditions of  Theorem~\ref{exp cong}(i) are valid
(for more details of these equivalences see the proof of~\cite[Lemma 11]{bufl99}).

It thus suffices to show  that an element
$\mathcal{E} = \sum_{\psi \in \Ir(G)}\mathcal{E}_\psi e_{\psi}$ of $\Nrd_{\QQ_p[G]}\bigl(\mathfrak{M}_p^\times\bigr)\subset \bigoplus_{\psi \in \Ir(G)}\QQ_p(\psi)$
belongs to $\ker(\delta_{G,p})$ if and only if the congruences in \eqref{first congs} are valid with each term $\mathcal{Q}_\psi$ replaced by $\mathcal{E}_\psi$ and
with $\ZZ_{(p)}$ replaced by $\ZZ_p$.
But, from the results of Lemma~\ref{general breuning}(i) and (ii) below, one has  $\delta_{G,p}(\mathcal{E}) = 0$ if and only if
$\res(\mathcal{E}) \in \ZZ_p[P]^\times$, and by Lemma~\ref{general breuning}(iii) this is true if and only if the congruences in~\eqref{first congs} are valid after
making the changes described above.

This therefore completes the proof of Theorem~\ref{exp cong}.
\end{proof}

\begin{lemma}\label{last one}
 There exists an isomorphism of $\ZZ_p[G]$-modules $\iota_p: A(F)^*_p \to A(F)_p$ such that $\calLstar_{A,F/k,j,\iota_p}$ is equal to the element $\mathcal{L}' := j_*(\sum_{\psi \in \Ir(G)} \mathcal{Q}_\psi e_\psi)$ defined above.
\end{lemma}

\begin{proof}
We claim first that for each $\rho \in \Ir(G)$ there is an equality
\begin{equation}\label{gauss sums_eq}
  \frac{\tau^* \bigl(\QQ,\Ind_k^\QQ\rho\bigr)}{w_\infty(\rho)} =
  \begin{cases}
     u_{{\eins}_G}\,\sqrt{\vert d_k\vert } = (-1)^{|S_{\rm r}|}\,\sqrt{\vert d_k\vert },        &\text{if $\rho = {\eins}_G$}\\
     u_\epsilon\,\sqrt{\vert d_{K}/d_k\vert } = (-1)^{|S_{\rm r}'|}\,\sqrt{\vert d_K/d_k\vert },&\text{if $\rho = \epsilon$}\\
     u_\rho\,\sqrt{\vert d_{K}\vert \,Nf(\chi)},                                    &\text{if $\rho = \Ind_P^G\chi$}
\end{cases}
\end{equation}
where $\chi\in \Ir(P)\setminus\{\eins_P\}$ and with $S_{\rm r}'$ denoting the subset of $S_{\rm r}$ comprising places which split in $K/k$.

To prove this we note that for each $\rho$ in $\Ir(G)$ one has
\begin{equation}\label{first step}
 u_\rho^{-1} \cdot \tau^*\bigl(\QQ,\Ind_k^\QQ\rho\bigr) = \tau\bigl(\QQ,\Ind_k^\QQ\rho\bigr) = \Bigl(i^{|S_\CC|}\,\sqrt{\vert d_k\vert }\Bigr)^{\rho(1)}\cdot\tau(k,\rho),
\end{equation}
where the first equality follows straight from the definition~\eqref{ggs def} and the second from the result of~\cite[Theorem 8.1(iii)]{martinet}.

Now if $\rho = \eins_G$, then $\tau(k,\rho) = 1$ and $V_\rho^{I_w} = \CC_p$ for all $v$ in $S_{\rm r}$ so  $u_\rho = (-1)^{|S_{\rm r}|}$, whilst it is clear that $w_\infty(\rho) = i^{|S_\CC|}$, and so in this case the equality~\eqref{gauss sums_eq} is an immediate consequence of~\eqref{first step}.

Next we note that $\epsilon = \Ind_P^G({\eins}_P) - {\eins}_G$ and hence that~\cite[Theorem 8.1(iii)]{martinet} implies
\[ \tau(k,\epsilon) = \frac{\tau\bigl(k,\Ind_P^G({\eins}_P)\bigr)}{\tau(k,{\eins}_G)} = \tau\bigl(k,\Ind_P^G({\eins}_P)\bigr) = i^{|S_\RR^{\rm r}|}\sqrt{d_{K/k}} = i^{|S_\RR^{\rm r}|}\cdot \frac{\sqrt{\vert d_{K}\vert }}{\vert d_k\vert },\]
where $S_\RR^{\rm r}$ is the set of real places of $k$ that ramify in $K/k$ and $d_{K/k}$ the absolute norm of the different of $K/k$.
One then obtains the claimed equality~\eqref{gauss sums_eq} by substituting this formula for $\tau(k,\epsilon)$ into~\eqref{first step} and then using that fact that $w_\infty(\epsilon) = i^{|S_\CC|+ |S_\RR^{\rm r}|}$ whilst $u_\epsilon = (-1)^{|S'_{\rm r}|}$ since $\det\bigl(-\Fr_w\mid V_\epsilon^{I_w}\bigr)$ is equal to $-1$ for $v\in S_{\rm r}'$ and to $1$ for $v\in S_{\rm r}\setminus S_{\rm r}'$.

Finally we assume that $\rho = \Ind_P^G\chi$ with $\chi\in \Ir(P)\setminus\{{\eins}_P\}.$ Then, just as above, one finds that
\begin{align*} \tau(k,\rho) = &\tau(k,\Ind_P^G({\eins}_P))\cdot \tau(K,\chi) \\
= &i^{|S_\RR^{\rm r}|}\,\frac{\sqrt{\vert d_{K}\vert }}{\vert d_k\vert }\,\tau(K,\chi) \\
= &i^{|S_\RR^{\rm r}|}\,\frac{\sqrt{\vert d_{K}\vert }}{\vert d_k\vert }\,\frac{W(\check{\chi})}{W_\infty(\chi)}\,\sqrt{Nf(\chi)},\end{align*}
where $W(\check{\chi})$ is the Artin root number of $\check\chi$ and $W_\infty(\chi)$ the infinite part of the Artin root number of $\chi$ and the final equality follows from the very definition of $\tau(K,\chi)$. Now $W_\infty(\chi) = 1$ since no archimedean place ramifies in $F/K$.
In addition, the inductivity of Artin root numbers implies $W(\check\chi) = W(\check\rho)$ and so, since $\check\rho$ is an orthogonal character, the main result of Fr\"ohlich and Queyrut in~\cite{fq} implies that $W(\check\chi)=1$. Thus, upon substituting the last displayed expression into~\eqref{first step}, and noting that $w_\infty(\rho) = i^{\rho(1)|S_\CC| + |S_\RR^{\rm r}|}$, one obtains the claimed equality~\eqref{gauss sums_eq} in this case.

Having proved~\eqref{gauss sums_eq}, we now write $\theta_Q$ for the generator of the $\ZZ_p[G]$-module $A(F)_p^*$ that sends $Q$ to $1$ and $g(Q)$ to zero for each non-trivial element $g$ of $P$. We then define $\iota_p$ to be the isomorphism of $\ZZ_p[G]$-modules $A(F)^*_p \to A(F)_p$ that sends $\theta_Q$ to $Q$.

One computes that $\bigl((\CC_p\otimes_{\ZZ_p}\iota_p)\circ \lambda^{{\rm NT},j}_{A,F}\bigr)(Q)= j_*(R_Q)(Q)$, where $R_Q$ is the resolvent element $\sum_{g \in P} \bigl\langle Q,g(Q)\bigr\rangle_F\cdot g $ in $\CC_p[G]$, and also that $\theta_Q\bigl(\iota_p(\vartheta)\bigr) = \vartheta(Q)$ for all $\vartheta\in A(F)_p^*$.

Proposition~\ref{Q_prop} implies $r_{\rho_A} = 0$ so $R_{\rho_A,\iota_p}^{{\rm NT},j}(A_{/F}) = 1 = H_{\rho_A}$ and also that for each $\psi \in \Ir(G)\setminus \{\rho_A\}$, the complex vector space $(V_\psi \otimes_\ZZ A(F))^G = \CC[G]\cdot T_{\psi}(Q)$ has dimension one and hence $ R_{\psi,\iota_p}^{{\rm NT},j}(A_{/F}) \cdot T_{\psi}(Q) = \bigl((\CC_p\otimes_{\ZZ_p}\iota_p) \circ \lambda^{{\rm NT},j}_{A,F}\bigr)\bigl(T_{\psi}(Q)\bigr).$

Now, for any $\vartheta\in A(F)_p^*$ and  $P\in A(F)$ one has $T_{\check{\psi}}(\vartheta)(P) = \vartheta\bigl( T_{\check\psi}^{\#} (P) \bigr)$ where we write  $x\mapsto x^{\#}$ for the $\ZZ_p$-linear involution on $\ZZ_p[G]$ that inverts elements of $G$. In addition,
for each $\psi\neq \rho_A$ one has
\begin{align*}
  T_{\check\psi}(\theta_Q)\Bigl(T_{\psi}(Q)\Bigr) &=
  \sum_{g \in G}\sum_{h \in G} \check{\psi}(g^{-1})\, \psi(h^{-1})\, \theta_Q\Bigl( g^{-1} \, h\, Q \Bigr) \\
  &= \sum_{g\in G} \check\psi(g^{-1})\psi(g^{-1}) +  (-1)^{1-\rk(A_k)}\cdot\sum_{g\in G} \check\psi(g^{-1}) \, \psi( \tau g^{-1} )
\end{align*}
because $g^{-1}hQ$ is a multiple of $Q$ only when $h=g$ or $h=g\tau$. The first sum here is always equal to $\vert G\vert$, while the second is equal to $\vert G\vert$ for $\psi =\eins$, to $-\vert G\vert$ for $\psi = \epsilon$ and to $0$ otherwise.
Hence, writing $\vartheta$ for $\lambda^{{\rm NT},j}_{A,F}\bigl(T_\psi(Q)\bigr)$, we find that

\begin{align*}
   \frac{2|G|}{\psi(1)} \,R_{\psi,\iota_p}^{{\rm NT},j}&(A_{/F}) = T_{\check\psi}(\theta_Q)\Bigl( (\CC_p\otimes_{\ZZ_p}\iota_p)\circ \lambda^{{\rm NT},j}_{A,F} \bigl(T_\psi (Q) \bigr)\Bigr)\\ &=  \theta_Q\Bigl( T_{\check\psi}^{\#} \bigl( (\CC_p\otimes_{\ZZ_p}\iota_p) (\vartheta) \bigr) \Bigr)
  =  \theta_Q\Bigl( (\CC_p\otimes_{\ZZ_p}\iota_p) \bigl( T_{\check\psi}^{\#} (\vartheta) \bigr) \Bigr)\\
   &= \bigl( T_{\check\psi}^{\#} (\vartheta) \bigr) (Q)
  = \vartheta \bigl(T_{\check\psi}(Q)\bigr)
   = j(\bigl\langle T_{\psi}(Q), T_{\check \psi}(Q) \bigr\rangle_F) = \frac{2|G|}{\psi(1)} \,j(h_{F,\psi}(Q)). \label{second step}
\end{align*}
Upon substituting this equality and~\eqref{gauss sums_eq}
into the definition of $\calLstar_{A,F/k,j,\iota_p}$ one obtains the element $\mathcal{L}'$, as required.
\end{proof}

In the next result we write $\mathcal{O}_L$ for the valuation ring of a finite extension $L$ of $\QQ_p$.
\begin{lemma}\label{general breuning}
 Write $\mathfrak{M}_p'$ for the integral closure of $\ZZ_p$ in $\QQ_p[P]$.
\begin{enumerate}
\item[(i)]\label{general breuning_i}
  The following diagram commutes
 \begin{equation*}\label{gen br}
  \xymatrix@C+2ex{ \Nrd_{\QQ_p[G]}\bigl(\mathfrak{M}_p^\times\bigr) \ar[r]^{\res} & \mathfrak{M}'^\times_p\\
    \im\Bigl(\alpha: K_1\bigl(\mathfrak{M}_p\bigr) \to K_1\bigl(\QQ_p[G]\bigr)\Bigr) \ar[r]^{\res'} \ar[u]^{\Nrd_{\QQ_p[G]}} \ar[d]_{\partial_{G,p}'} &
    \im\Bigl(\beta: K_1\bigl(\mathfrak{M}'_p\bigr) \to K_1\bigl(\QQ_p[P]\bigr)\Bigr) \ar[u]_{ \Nrd_{\QQ_p[P]}}   \ar[d]^{\partial_{P,p}'}\\
    K_0\bigl(\ZZ_p[G],\QQ_p[G]\bigr)_{\rm tor} \ar[r]^{\res^G_{P,0}} & K_0\bigl(\ZZ_p[P],\QQ_p[P]\bigr)_{\rm tor}.}
 \end{equation*}
 Here $\res$ is the homomorphism that sends an element $\sum_{\psi \in \Ir(G)} \xi_\psi e_\psi$ in  $\Nrd_{\QQ_p[G]}(\mathfrak{M}_p^\times) \subset \zeta\bigl(\CC_p[G]\bigr)^\times = \sum_{\psi}\CC_p^\times e_\psi$ to
 $\xi_{\eins_G}\xi_{\epsilon} e_{\eins_P} + \sum_{\chi\in \Ir(P)\setminus\{\eins_P\}}\xi_{\Ind_P^G\chi}e_\chi$, next $\alpha$ and $\beta$ are the natural scalar extension homomorphisms, $\res'$ and $\res^G_{P,0}$ the natural restriction homomorphisms and $\partial_{G,p}'$ and $\partial_{P,p}'$ the restrictions of the connecting homomorphisms $\partial_{G,p}$ and $\partial_{P,p}$.
\item[(ii)]\label{general breuning_ii}
 The homomorphism $\res^G_{P,0}$ is injective.
\item[(iii)]\label{general breuning_iii}
  Fix $\mathcal{E} := \sum_{\chi \in \Ir(P)}\mathcal{E}_\chi e_\chi$ inside $ \mathfrak{M}'^\times_p = \sum_{\chi\in \Ir(P)}\mathcal{O}_{\QQ_p(\chi)}^\times e_\chi$. Then $\mathcal{E}$ belongs to $\ZZ_p[P]^\times$ if and only if $\sum_{\chi \in \Ir(P)}\chi(\pi)^{-1}\mathcal{E}_\chi$ belongs to $|P|\cdot \ZZ_p$ for all $\pi$ in $P$.
\end{enumerate}
\end{lemma}
\begin{proof} To prove claim (i) we fix a set of representatives $\Ir(P)^\dagger$ of the orbits of the action of $G_{\QQ_p}$ on $\Ir(P)$ and
abbreviate the functor $\Ind^G_P(-)$ to $\Indshort^G_P(-)$.

The commutativity of the lower square of the diagram follows from the naturality of the long exact sequences of relative $K$-theory. To consider the upper square we
use the $\QQ_p$-algebra isomorphisms
$$\omega_G:\QQ_p[G] \cong \QQ_p\times \QQ_p\times \prod_{\chi \in \Ir(P)^\dagger\setminus \{\eins_P\}}{\rm M}_2\bigl(\QQ_p(\Indshort^G_P(\chi))\bigr)$$
  and
$\omega_P: \QQ_p[P] \cong \prod_{\chi \in \Ir(P)^\dagger}\QQ_p(\chi)$ where  $\omega_G(g) = \bigl(1,\epsilon(g),(S_{\Indshort_P^G(\chi)}(g))_{\chi} \bigr)$,
with $S_{\Indshort_P^G(\chi)}$ a representation $G \to {\rm M}_2\bigl(\QQ_p(\Indshort^G_P(\chi))$ of character $\Indshort_P^G(\chi)$, and
$\omega_P(\pi) = (\chi(\pi))_\chi$ for each $g \in G$ and $\pi \in P$. Taken together these isomorphisms induce a diagram

\[
\xymatrix@R+2ex{ \QQ_p^\times \times \QQ_p^\times \times \prod\limits_{\substack{\chi \in \Ir(P)^\dagger\\ \chi\neq \eins_P } } \QQ_p(\Indshort^G_P(\chi))^\times \ar[r]^{\widetilde{\res}} &
  \prod\limits_{\chi \in \Ir(P)^\dagger} \QQ_p(\chi)^\times\\
 K_1(\QQ_p)\times K_1(\QQ_p)\times \prod\limits_{\substack{\chi\in \Ir(P)^\dagger\\ \chi\neq \eins_P}} K_1\Bigl ({\rm M}_2\bigl(\QQ_p(\Indshort^G_P(\chi))\bigr)\Bigr) \ar[u]^(0.4){(\det,\det,(\det)_\chi)} &
  \prod\limits_{\chi \in \Ir(P)^\dagger} K_1(\QQ_p(\chi))\ar[u]_(0.4){ (\det)_\chi  }\\
 K_1\bigl(\QQ_p[G]\bigr) \ar[r]^{\res^G_{P,1}} \ar[u]^(0.35){K_1(\omega_G)} & K_1\bigl(\QQ_p[P]\bigr)\ar[u]_(0.4){ K_1(\omega_P) } \\
}\]
in which the left and right hand composite vertical arrows are equal to $\Nrd_{\QQ_p[G]}$ and $\Nrd_{\QQ_p[P]}$ and $\widetilde{\res}$ is defined to make the diagram
commute. To complete the proof of claim (i) it thus suffices to show $\widetilde{\res}$ sends each element
$(\alpha,\tilde\alpha,(a_\chi)_\chi)$ to $(\alpha\tilde\alpha,(a_\chi)_\chi)$ and this follows from the argument used by Breuning in~\cite[Lemma 3.9]{Breuning}.

To prove claim (ii) we recall a group is said to be $\QQ_p$-elementary if it is isomorphic to a group $(\ZZ/n\ZZ) \rtimes J$ where $J$ is an
$\ell$-group for some prime $\ell$ that is coprime to $n$ and the image of the homomorphism $J \to \Aut (\ZZ/n\ZZ) \cong \Gal(\QQ(e^{\frac{2\pi i}{n}})/\QQ)$ belongs
to the decomposition subgroup of $p$.
This means that a subgroup of $G$ is $\QQ_p$-elementary if it is either a subgroup of $P$ or of the form $\langle C,\tau'\rangle$ where $C$ is a cyclic subgroup of
$P$ and $\tau'$ an element of order $2$. We  consider the exact commutative diagram
\[ \xymatrix{
K_1\bigl(\ZZ_p[G]\bigr) \ar[r] \ar[d]^{\alpha_1} & K_1\bigl(\QQ_p[G]\bigr) \ar[r]\ar[d]^{\alpha_2} &  K_0\bigl(\ZZ_p[G],\QQ_p[G]\bigr) \ar[r] \ar[d]^{\alpha_3} & 0\\
\varprojlim_H K_1\bigl(\ZZ_p[H]\bigr) \ar[r] & \varprojlim_H K_1\bigl(\QQ_p[H]\bigr) \ar[r] & \varprojlim_H K_0\bigl(\ZZ_p[H],\QQ_p[H]\bigr)
}\]
where the limits are over all $\QQ_p$-elementary subgroups $H$ of $G$. The transition maps are the homomorphisms induced by inclusions $H \subseteq H'$ and by maps of
the form $H \to gHg^{-1}$ for $g \in G$ and all vertical arrows are the natural restriction maps.
By a theorem of Dress~\cite{dress} (see also~\cite[Theorem 11.2]{oliver}), the maps $\alpha_1$ and $\alpha_2$ are bijective and hence $\alpha_3$ is injective. Now
$K_0\bigl(\ZZ_p[H],\QQ_p[H]\bigr)_{\rm tor}$ is trivial whenever $H$ has order prime to $p$ and in~\cite[Proposition 3.2.(2)]{Breuning} Breuning has shown that the
restriction map $K_0\bigl(\ZZ_p[H],\QQ_p[H]\bigr)_{\rm tor} \to K_0\bigl(\ZZ_p[C],\QQ_p[C]\bigr)_{\rm tor}$ is injective for any subgroup $H$ of the form
$\langle C,\tau'\rangle$.
The map $\alpha_3$ therefore restricts to give an injective homomorphism
\[ K_0\bigl(\ZZ_p[G],\QQ_p[G]\bigr)_{\rm tor} \to \varprojlim_J K_0\bigl(\ZZ_p[J],\QQ_p[J]\bigr)_{\rm tor} \cong K_0\bigl(\ZZ_p[P],\QQ_p[P]\bigr)_{\rm tor},\]
where $J$ runs over all subgroups of $P$, as required to prove claim (ii).

Claim (iii) follows from the equalities
\[
 \mathcal{E} = \sum_{\chi \in \Ir(P)}\mathcal{E}_\chi\sum_{\pi \in P}|P|^{-1}\chi(\pi)^{-1} \pi
 = \sum_{\pi \in P}\biggl( |P|^{-1}\sum_{\chi \in \Ir(P)} \chi(\pi)^{-1}\mathcal{E}_\chi \biggr)\pi
\]
and the fact that $\ZZ_p[P]^\times = \ZZ_p[P]\cap \mathfrak{M}_p'^\times$.
\end{proof}

\section{Special cases}\label{special cases}

In this section we use the criteria of Theorem~\ref{exp cong} to give both theoretical and numerical verifications of the $p$-part of the equivariant Tamagawa number
conjecture for pairs  $\bigl(h^1(A_{/F})(1),\ZZ[G]\bigr)$ where $A$ is an elliptic curve for which $A(F)$ has strictly positive rank and $G$ is both non-abelian and of
order divisible by $p$.

We believe that, apart from the recent results of Bley in~\cite{Bley3}, where the group $A(F)_p$ is assumed to be trivial and the field $F$ to be an abelian extension
of $\QQ$ of exponent $p$, these results constitute the first verifications of the $p$-part of the equivariant Tamagawa number conjecture for any elliptic curve and any
Galois extension of degree divisible by $p$.

We begin by giving the proof of Theorem \ref{HPintro}, which relies on the theory of Heegner points and makes crucial use of the theorem of Gross and Zagier.
We note that the additional hypotheses in Theorem~\ref{HPintro} imply the validity of hypothesis~\ref{hyp_i}. In addition,  Kolyvagin~\cite[Proposition 2.1]{gross_koly}
shows that in this case one has $\sha_p(A_{K})=0$.
In particular one knows that the $p$-primary part of the Birch and Swinnerton-Dyer conjecture holds for $A_{/K}$.
Furthermore, the hypotheses~\ref{hyp_d} and \ref{hyp_h} are obviously satisfied in the setting of Theorem \ref{HPintro}.

\subsection{The proof of Theorem \ref{HPintro}}\label{proofmain}

In view of Theorems~\ref{equivalence theorem} and~\ref{exp cong} we are reduced to verifying the conditions (i) and (ii) that occur in the latter  result.

To do this we fix a modular parametrisation $\varphi\colon X_0(N)\to A$ of smallest degree. We also denote by $c$ the Manin constant of $\varphi$ and write $Q$ for the
trace in $A(F)$ of the Heegner point that is defined over the Hilbert class field of $K$.

Set $C := 4 \,c_{\infty} \cdot c^{-2}\cdot \bigl\vert\mathcal{O}_{K}^\times\bigr\vert^{-2} $ and $c_{\infty} = \bigl[A(\RR):A^0(\RR)\bigr]$ is the number of connected
components of $A(\RR)$.
Under our hypotheses, the theorem of Gross and Zagier (see, in particular, \cite[\S I, (6.5) and the discussion on p. 310]{GZ}) implies that for each $\chi\in \Ir(P)$
one has
\begin{align}\label{GZ form}
0 \neq \frac{ L'(A_{/K},\chi,1) \sqrt{\vert d_K\vert}}{\Omega^{+}_\infty(A) \Omega^{-}_\infty(A)} &= C\cdot \frac{1}{|P|} \bigl\langle T_\chi(Q),T_{\check\chi}(Q)\bigr\rangle_F\\
 &= \begin{cases} C\cdot  h_{F,\Ind^G_P(\chi)}(Q),\quad &\text{if $\chi \not= \eins_P$}\\
                  C \cdot h_{F,\rho_A'}(Q), &\text{if $\chi = \eins_P$,}
    \end{cases} \notag
\end{align}
where $\rho'_A$ is the linear character of $G$ appearing in $A(K)$ and $\infty$ denotes the archimedean place of $\QQ$.
The second equality here follows from the equality $T_{\Ind_G^P\chi} = T_\chi + T_{\check\chi}$ and the fact that $\langle T_{\chi}(Q), T_{\chi}(Q)\rangle_F = 0$
(since the height pairing is $G$-invariant).

For each $\psi$ in $\Ir(G)$ we set
\begin{equation}\label{Qhat_eq}
  \hat{\mathcal{Q}}_{\psi} := \frac{ \Lstar(A,\psi,1)\cdot \sqrt{ d_{\psi} } }{ \Omega(A,\psi) \cdot H_{\psi} }
\end{equation}
where the quantities $d_\psi$ and $H_{\psi}$ are as defined just before Theorem~\ref{exp cong}. We also write
\begin{equation}\label{t_def_eq}
  t_\psi = \frac{ \Lstar_{S_{\rm r}} (A,\psi, 1) }{\Lstar(A,\psi,1)}
\end{equation}
for the correction term that accounts for the $S_{\rm r}$-truncation in the leading terms for each $\psi\in\Ir(G)$.
Then the non-zero complex number $\mathcal{Q}_{\psi} := u_\psi\cdot t_\psi \cdot \hat{\mathcal{Q}}_\psi$ is as in Theorem~\ref{exp cong} using our Heegner point $Q$.
By using~\eqref{GZ form}, and the fact that $d_\psi = d_{\eins} d_\epsilon = \vert d_K \vert$ for all $\psi$ of dimension two (as $F/K$ is unramified), one then finds
that
\begin{equation}\label{def char2}
\begin{alignedat}{2}
  \mathcal{Q}_{\eins} \cdot \mathcal{Q}_{\epsilon} &=  u_{\eins} \,t_{\eins}\, u_\epsilon \, t_\epsilon\cdot  C, \qquad &&\text{and }\\
  \mathcal{Q}_{\psi} & = u_\psi\, t_\psi \cdot C \qquad &&\text{if $\psi =\Ind^G_P\chi$ for $\chi\in \Ir(P)\setminus\{\eins_P\}$.}
\end{alignedat}
\end{equation}
Our hypotheses imply that $A(K)$ has rank one and $\sha_p(A_{K})$ vanishes by Kolyvagin~\cite{gross_koly} and hence all of the hypotheses of Theorem~\ref{exp cong} are
satisfied except for the requirement that $p$ splits in $K/\QQ$.
However, the sole purpose of the latter hypothesis is to ensure (via the proof of Proposition~\ref{Q_prop}) that the $\QQ[P]$-module $\QQ\otimes_\ZZ A(F)$ has a free
rank one direct summand and so, since~\eqref{GZ form} implies that the point $Q$ generates such a summand, this hypothesis can be ignored in our case.
Following Remark~\ref{gen exp cong}, it therefore suffices for us to prove that the terms $\mathcal{Q}_\psi$ satisfy both the conditions of Theorem~\ref{exp cong}(i) and
the congruences~\eqref{first congs}.

Since the $p$-part of the Birch and Swinnerton-Dyer conjecture holds for $A_{/K}$, the hypotheses~\ref{hyp_a} and~\ref{hyp_b} and the known vanishing of $\sha_p(A_{K})$
combine to imply that $\hat{\mathcal{Q}}_{\eins}$ and $\hat{\mathcal{Q}}_{\epsilon}$ are both $p$-units.
The hypothesis~\ref{hyp_c} implies that $A$ has good reduction at $p$ and so the hypothesis (ii) combines with ~\cite[Theorem 2.7]{manin_constant} to imply that
the Manin constant $c$ is a $p$-unit. Using also the fact that $p$ is
unramified, we find that
$C$ is a $p$-unit too.

We proceed to compute the values of $u_{\psi}$ and $t_\psi$. Since $F/K$ is unramified, we have that each ramified place $v$ in $S_{\rm r}$ has ramification index $2$
and $v$ does not split in $K/\QQ$. It is easy to see that $u_\epsilon = 1$ and that $u_{\eins} = u_\psi = (-1)^{\vert S_{\rm r}\vert}$ for all $\psi$ of dimension two. Next, for
any $\psi$ in $\Ir(G)$, we have
\begin{equation*}
  t_\psi = \prod_{v\in S_{\rm r}} \Nrd_{\CC_p[G]}\Bigl( 1 - \Fr_w^{-1} \vert \kappa_v\vert ^{-1} \Bigl\vert \bigl( T_p (A) \otimes V_\psi\bigr)^{I_w} \Bigr)
\end{equation*}
and since no place of bad reduction is allowed to ramify by hypothesis~\ref{hyp_e}, one has $\bigl( T_p (A) \otimes V_\psi\bigr)^{I_w} = T_p (A) \otimes V_{\psi}^{I_w}$.
Hence $t_\epsilon = 1$ and $t_{\eins} = t_\psi$ for all $\psi$ of dimension two. Moreover, this last value $t_{\eins}$ is equal to
$\prod_{v \in S_{\rm r}} \vert A(\kappa_v)\vert/ \vert \kappa_v\vert$, which is a $p$-unit by hypothesis~\ref{hyp_f}. Therefore $\mathcal{Q}_\psi$ is a $p$-unit for
all $\psi\in\Ir(G)$.

Next we note that $\mathcal{Q}_{\eins} \cdot \mathcal{Q}_\epsilon = \mathcal{Q}_\psi$ for any $\psi$ of dimension two as we have shown that
$u_{\eins} \, t_{\eins}\, u_\epsilon\, t_{\epsilon} = u_\psi\,t_\psi$. Since each element $\mathcal{Q}_\psi$ also belongs to $\QQ$, this formula makes it clear that
they satisfy the condition of Theorem~\ref{exp cong}(i). In addition, it shows that for any $\pi \in P$ one has
$$
-\sum_{\chi\in\Ir(P)\setminus\{\eins_P\}} \check\chi(\pi) \mathcal{Q}_{\Ind^G_P\chi} = - \mathcal{Q}_{\eins} \, \mathcal{Q}_{\epsilon} \cdot \Bigl( -1 + \sum_{\chi\in\Ir(P)} \check\chi(\pi) \Bigr).
$$
Finally we note that the last sum in the above expression is equal to $0$ if $\pi\neq 1$ and to $\vert P\vert$ otherwise. Hence the expression is in both cases
congruent to $\mathcal{Q}_{\eins}\,\mathcal{Q}_\epsilon$ modulo $p\ZZ_{(p)}$, as required to complete the proof of Theorem \ref{HPintro}.


\subsection{\texorpdfstring{$S_3$}{S3}-extensions}\label{Sthree}
 In this section we investigate the case that $F/K$ is a cyclic extension of degree $p=3$ and hence $F/k$ is a non-abelian extension of degree six. We show that the conjecture of Birch and Swinnerton-Dyer (or `BSD' for short) implies an explicit congruence modulo rational squares and then combine this congruence with Theorem~\ref{exp cong} to prove the equality~\eqref{etnc equiv} for a natural family of examples.

We assume throughout that $G$ is a non-abelian group of order six. In this case $\Ir(G)$ comprises $\eins := \eins_G$, $\epsilon$ and $\psi = \Ind_P^G\chi$ for a fixed $\chi$ in $\Ir(P)\setminus\{\eins_P\}$. We fix a subfield $L$ of $F$ that is of degree $3$ over $k$ and set $H := G_{F/L}$.
\subsubsection{An arithmetic congruence}
For each character $\eta$ in $\Ir(G)$ we define
\begin{equation}\label{Qtilde_eq}
  \tilde{\mathcal{Q}}_{\eta} := \frac{ \Lstar (A,\eta,1)\cdot \sqrt{d_{\eta}}}{{\Omega(A,\eta)}\cdot \tilde{H}_{\eta}}.
\end{equation}
Here the quantities $d_{\eta}$ are as defined just before Theorem~\ref{exp cong} and we have set
\begin{equation}\label{Htilde_eq}
  \tilde{H}_{\eins} := \Reg(A_{/k}),\qquad \tilde{H}_{\epsilon} := \frac{\Reg(A_{/K})}{\Reg(A_{/k})}, \qquad\text{and}\qquad
  \tilde{H}_\psi := \frac{\Reg(A_{/L})}{\Reg(A_{/k})}
\end{equation}
with $\Reg(A_{/E})$ denoting the regulator of $\langle\cdot,\cdot \rangle_{E}$ on the Mordell-Weil group $A(E)$ for any field $E$.
These slight variations of the regulators that we used earlier fit well with BSD and do not require any knowledge of the explicit Galois structure of  $A(F)_p$. In fact,  BSD predicts that each expression $\tilde{\mathcal{Q}}_{\eta}$ is a non-zero rational number.

In the following result we include the assumed analytic continuation of the $L$-series and finiteness of the Tate-Shafarevich group in the assumption that BSD holds but do not assume any of the hypotheses~\ref{hyp_a}--\ref{hyp_d} and~\ref{hyp_f}--\ref{hyp_h}.
This result may therefore suggest one sort of congruence relation that might be expected to hold when our hypotheses fail. We note also that the remark made by Dokchitser and Dokchitser in the fourth paragraph after Conjecture~1.4 of~\cite{dok_nonab} hints at the possibility of this sort of result in a more restrictive setting.

\begin{theorem}\label{arith_thm}
  Let $A$ be an elliptic curve over $k$ and assume no place at which $A$ has bad reduction ramifies in $F/k$. Then if the Birch and Swinnerton-Dyer conjecture holds for $A_{/k}$, $A_{/K}$ and $A_{/L}$  one has a congruence modulo non-zero rational squares
  \begin{equation}\label{tilde_congruence_eq}
   \tilde{\mathcal{Q}}_{\eins} \cdot \tilde{\mathcal{Q}}_{\epsilon} \equiv \tilde{\mathcal{Q}}_{\psi}\pmod{\square}.
  \end{equation}
\end{theorem}
\begin{proof}
 Fix a finite field extension $E$ of $k$. If $v$ is a (finite or infinite) place of $E$, write $c(A/E_v)$ for the number of connected components $\bigl[ A(E_v) :A(E_v)^0 \bigr]$ of $A(E_v)$. For a finite place $v$ in $E$, also write $\omega_v^{\text{N\'e}}$ for a N\'eron differential of $A_{/E_v}$. Then BSD for the field $E$ asserts that
  \begin{equation}\label{BSD statement}
    \frac{ \Lstar (A_{/E},1) \, \sqrt{\vert d_{E}\vert }}{\Omega(A_{/E}) \cdot \Reg(A_{/E})} =  \prod_{\text{all }v} c(A/{E_v}) \cdot \prod_{\text{finite }v} \biggl\vert \frac{\omega_A}{\omega_v^{\text{N\'e}}}\biggr\vert_v \cdot \frac{\vert \sha(A_E)\vert}{\vert A(E)_{\rm tor}\vert^2}
  \end{equation}
 where we write $\Omega(A_{/E})$ for the period $\prod_{v\in S_\RR^E} \Omega_v^+(A) \cdot \prod_{v\in S_\CC^E} \Omega_v(A)$ of $A$ over $E$, as defined in~\S\ref{archi_subsubsection} with respect to a fixed invariant differential $\omega_A$ of $A_{/k}$.

The term $\tilde{\mathcal{Q}}_{\eins}$ is equal to the left hand side of~\eqref{BSD statement} with $E=k$. Furthermore the products $\tilde{\mathcal{Q}}_{\eins}\tilde{\mathcal{Q}}_{\epsilon}$ and $\tilde{\mathcal{Q}}_{\eins}\tilde{\mathcal{Q}}_{\psi}$ link to the left hand sides of~\eqref{BSD statement} for the fields $K$ and $L$ respectively. More precisely, one has
$$
 \tilde{\mathcal{Q}}_{\eins} \cdot \tilde{\mathcal{Q}}_\epsilon = \frac{ \Lstar (A_{/K},1) \, \sqrt{\vert d_{K}\vert }}{\Omega(A_{/K}) \cdot \Reg(A_{/K})} \cdot \prod_{v \in S_{\RR}^{\rm r}} c(A/k_v)
$$
where, as before, $S_{\RR}^{\rm r}$ denotes the set of real places of $k$ that become complex places in $K$. The formula for $\tilde{\mathcal{Q}}_{\eins}\tilde{\mathcal{Q}}_{\psi}$ is modified by the same factor as there is exactly one complex place in $L$ above each place in $S_{\RR}^{\rm r}$.
In proving this last formula one also uses the fact that $d_\psi = \vert d_K\vert\cdot Nf(\chi)$ is equal to $\vert d_L/d_k\vert $ since $\psi = \Ind_H^G\eins_H - \eins$.

In addition, since we are working modulo squares, we may neglect the terms $\vert \sha(A_E)\vert$ and $\vert A(E)_{\rm tor}\vert^2$ which occur on the right hand side of the formula~\eqref{BSD statement}. The required congruence~\eqref{tilde_congruence_eq} will therefore be proved if we can show for each place $v$ in $k$ that
\begin{equation}\label{tam_eq}
    c(A/k_v) \cdot \prod_{\substack{w\mid v\\ \text{in }K}} c(A/K_w) \equiv \prod_{\substack{w\mid v\\ \text{in }L}} c(A/L_w) \pmod{\square}
  \end{equation}
  and for each finite place $v$ in $k$ that
  \begin{equation}\label{ne_eq}
    \biggl\vert \frac{\omega_A}{\omega_v^{\text{N\'e}}}\biggr\vert_v \cdot
 \prod_{\substack{w\mid v\\ \text{in }K}} \biggl\vert \frac{\omega_A}{\omega_w^{\text{N\'e}}}\biggr\vert_w
  =  \prod_{\substack{w\mid v\\ \text{in }L}} \biggl\vert \frac{\omega_A}{\omega_w^{\text{N\'e}}}\biggr\vert_w .
  \end{equation}
 Now, by our assumption, no place at which $A$ has bad reduction is ramified in $F/k$ and so the N\'eron differential for $A$ over $k_v$ remains a N\'eron differential for $A$ over both of the fields $K_w$ and $L_w$. Hence the equation~\eqref{ne_eq} is valid for all finite places, because $\omega_A/\omega^{\text{N\'e}}_v$ is in $k_v$ and one has $[k:k]+[K:k] = [L:k]$.

 Next we note that the congruence~\eqref{tam_eq} only needs to be checked at places at which $A$ has bad reduction (and which therefore do not ramify in $F/k$) and at infinite places. If the decomposition group $G_v$ at a place above $v$ in $F$ is trivial, then both sides of this congruence are equal to $c(A/k_v)^3$.
 If $G_v$ is cyclic of order $2$, then there is one place $w'$ in $L$ with  $L_{w'} = k_v$ and one place $w''$ with $L_{w''}=K_w$ for the unique place $w$ above $v$ in $K$ and so both sides of~\eqref{tam_eq} are equal to $c(A/k_v)\cdot c(A/K_w)$. Finally we have to treat the case when $G_v$ is cyclic of order $3$ and hence the place $v$ is finite. In this case the left hand side of~\eqref{tam_eq} is equal to $c(A/k_v)^3$ while the right hand side is equal to $c(A/L_w)$ where $L_w$ is an unramified cubic extension of $k_v$.
 If $c(A/L_w)=c(A/k_v)$ then we have indeed a congruence modulo squares. However it is possible that the Tamagawa numbers change in an unramified extension. Luckily, the only possibility for this to happen in a cubic extension is when the Kodaira type is I${}_{0}^*$ and then the change is from $c(A/k_v)=1$ to $c(A/L_w)=4$, see for instance Step~6 in~\cite{sil2} on page~367, and the congruence~\eqref{tam_eq} holds in all cases.
\end{proof}

\subsubsection{The connection to eTNC$_p$}\label{etnc S_6}
We now use Theorem~\ref{arith_thm} to show that, under the hypotheses of Theorem~\ref{exp cong}, the relevant cases of BSD imply the equality~\eqref{etnc equiv}.

 \begin{corollary}\label{arith_cor}
   We assume that the elliptic curve $A$ and field $F$ satisfy the hypotheses~\ref{hyp_a}--\ref{hyp_h}. We assume also that $G$ and $L$ are as in Theorem~\ref{arith_thm} (so $p = 3$) and that the Birch and Swinnerton-Dyer conjecture holds for $A$ over each of the fields $k$, $K$ and $L$.

 Then the equality~\eqref{etnc equiv} is valid provided that $\sha_p(A_{K})=0$, $\rk (A_K) =1$ and there exists a point $Q$ in $A(F)$ which generates a $G$-module of finite prime-to-$p$ index in $A(F)$ that is isomorphic to $\ZZ[G]\bigl(1-(-1)^{\rk(A_{k})}\tau\bigr)$.
 \end{corollary}

\begin{proof}
  First, we note that given our hypotheses~\ref{hyp_a}, \ref{hyp_b} and~\ref{hyp_e}, our choice of $\omega_A$ (as in \S\ref{archi_subsubsection}) and our assumption that
  $\sha_p(A_{K})$ vanishes, the validity of BSD implies that the term $\tilde{\mathcal{Q}}_\eta$ is a $p$-adic unit for all $\eta\in\Ir(G)$.

  Next we link $\tilde{\mathcal{Q}}_\eta$ to $\mathcal{Q}_\eta$ in Theorem~\ref{exp cong}, the difference being the terms $\tilde{H}_\eta$ versus $H_{\eta}$ and the terms
  $u_\eta$ and $t_\eta$ (as defined in~\eqref{t_def_eq}). Using the explicit structure of $A(F)_p$, it is easy to show that
  \begin{equation*}
    H_{\psi} \equiv \tilde H_\psi, \qquad
    H_\epsilon \equiv \begin{cases}
                        \tilde{H}_{\epsilon}&\text{ if $\rho_A=\eins$}\\
                       \frac{1}{2}\, \tilde{H}_{\epsilon}&\text{ if $\rho_A=\epsilon$}
                     \end{cases}
\qquad\text{and}\qquad
    H_{\eins} \equiv \begin{cases}
                 \tilde H_{\eins} &\text{ if $\rho_A=\eins$}\\
                  2 \,\tilde H_{\eins} &\text{ if $\rho_A=\epsilon$.}
               \end{cases}
  \end{equation*}
  where all congruences are modulo squares in $\ZZ_{(3)}^{\times}$; namely the quotients are squares of indices, like the index of $\ZZ[G]Q$ in $A(F)$. Up to squares in
  $\ZZ_{(3)}^{\times}$ we therefore have a congruence

  \begin{equation*}
     \frac{1}{u_{\eins}\, t_{\eins}\,u_\epsilon\,t_{\epsilon} } \mathcal{Q}_{\eins} \, \mathcal{Q}_{\epsilon}
   \equiv
     \frac{1}{u_\psi\, t_\psi} \mathcal{Q}_\psi.
  \end{equation*}
  An argument similar to the one that concludes the proof of Theorem~\ref{HPintro} hence implies that we will have verified the criterion in Theorem~\ref{exp cong} if we
  show that $u_{\eins}\, t_{\eins}\,u_\epsilon\,t_{\epsilon}$ and $u_\psi\, t_\psi$ are $3$-adic units that are congruent modulo $3$.
  Writing $N_v = \vert A(\kappa_v)\vert$ for the number of points in the reduction at a place $v$ and $q_v = \vert \kappa_v\vert$, we can summarise the computations of
  the local contribution to these terms at a place $v\in S_{\rm r}$ in the following table according to the type of ramification. Here $e_v$ stands for the ramification
  index at a place in $F$ above $v$ and $f_v$ for the residual degree.
  \begin{center}\begin{tabular}{cc|ccc|ccc}
     & & \multicolumn{3}{c|}{$\det\bigl(-Fr_w\bigl\vert V_\eta^{I_w}\bigr)$} &
      \multicolumn{3}{c}{$\det\bigl(1-Fr_w^{-1}\, q_v^{-1}\bigl\vert (V_\eta\otimes T_p(A))^{I_w}\bigr)$} \\
    $e_v$ & $f_v$ & $u_{\eins}$ & $u_\epsilon$ & $u_\psi$ & $t_{\eins}$ & $t_\epsilon$ & $t_\psi$ \\ \hline &&&&&&&\\[-2ex]
    2     & 1     & $-1$      & 1            & $-1$    & $N_v/q_v$ & 1            & $N_v/q_v$ \\
    3     & 1     & $-1$      & $-1$         & 1       & $N_v/q_v$ & $N_v/q_v$    & 1 \\
    3     & 2     & $-1$      & 1            & 1       & $N_v/q_v$ & $N_w/q_w \cdot q_v/N_v$ & 1
  \end{tabular}\end{center}
  In the last line $w$ denotes the unique place $w$ in $K$ above $v$. If a place $v$ were totally ramified it must be above $2$ or $3$ as the tame inertia group is
  cyclic and it can not be above $2$ because the wild inertia group is normal in the inertia group. Hence there was no need to list the totally ramified case as all
  places above $p=3$ were assumed to be unramified by~\ref{hyp_h}.
  From the table we can conclude that all the terms $u_\eta$ and $t_\eta$ are indeed $3$-units by~\ref{hyp_f} and that
  $$
    \frac{u_{\eins}\,u_\epsilon}{u_\psi}\cdot \frac{t_{\eins}\,t_\epsilon}{t_\psi}\equiv \prod_{v \in S''_{\rm r}} \biggl(-\frac{N_w}{q_w} \biggr)
   \equiv \prod_{v \in S''_{\rm r}} \bigl(-N_w\bigr)\pmod{\square}
  $$
  with $S''_{\rm r} := \{v \,\vert \,e_v = 3\text{ and }f_v=2\}$.

  Hence we are reduced to showing that $-N_w\equiv 1\pmod{3}$ when $e_v=3$ and $f_v=2$. We first note that for such a place we must have $q_v\equiv -1\pmod{3}$. Indeed,
  this is true because the map $\theta_0: \Gal\bigr(F_{w'}/K_w\bigr) \to \mu_3(\kappa_w)$ in Corollaire~IV.1 in~\cite{corps_locaux} is $G$-equivariant and, since
  $F_{w'}/k_v$ is dihedral, the action of $\tau$ on $\mu_3(\kappa_w)$ is non-trivial.

  Finally, we have the equality $N_w = N_v\cdot ( 2q_v + 2 - N_v)$ valid for all quadratic extension of finite fields. Hence
  $N_w \equiv  - N_v^2 \equiv -1 \pmod{3}$.
\end{proof}

\begin{remark}
 A closer analysis of the above argument shows that the hypotheses of Corollary~\ref{arith_cor} may be weakened a little. One can allow places above $3$ to be tamely ramified in $F/k$ and can omit any assumption about the reduction of $A$ at such places. In addition, one need only assume that the group $A(\kappa_v)[p]$ vanishes for places $v$ that are both inert in $K/k$ and ramify in $F/K$.
 \end{remark}

\begin{remark}
 For any prime $p >3$, the methods used in the proofs of Theorem~\ref{arith_thm} and Corollary~\ref{arith_cor} enable one to deduce from the assumed validity of suitable cases of the Birch-Swinnerton-Dyer Conjecture a congruence for the product $\prod_{\dim(\psi) = 2} \hat{\mathcal{Q}}_\psi$, rather than for the sum $\sum_{\dim(\psi) = 2} \hat{\mathcal{Q}}_\psi$ that occurs in~\eqref{etnc equiv}.
\end{remark}

\subsection{Numerical examples}\label{numex}
In this final section we describe two numerical examples to further illustrate the predicted congruences in Theorem~\ref{exp cong} and to explain how one can check these congruences for numerous examples. It is comparatively straightforward to give examples with $p=3$ but Corollary~\ref{arith_cor} implies that there is limited interest in doing so. We therefore discuss examples with $p=7$ and $p=5$.

Our numerical computations were done using Sage~\cite{sage}, which uses underlying Pari-GP~\cite{pari}. The computations of the $L$-values was done in Magma~\cite{magma} which contains an implementation of~\cite{dokchitser}. The code can be obtained from the last named author's webpage.

\subsubsection{A Stark-Heegner point example}
We consider the example of the elliptic curve labelled 37a1 in Cremona's tables~\cite{cremona}
$$
A\colon\qquad y^2 \ +\  y \ = \ x^3\  -\  x
$$
over the Hilbert class field $F$ of $K=\QQ\bigl(\sqrt{577}\bigr)$. The curve has rank $1$ over $\QQ$ and $K$. The extension $F/K$ is of degree $p=7$ defined by a root $\xi$ of the polynomial
$$
x^7 - 2\, x^6 - 7\,x^5 + 10\, x^4 +13\,x ^3 -10\,x^2 - x + 1.
$$
All hypotheses~\ref{hyp_a}--\ref{hyp_h} except the finiteness of $\sha(A_{F})$ in~\ref{hyp_g} can be verified easily.
We find in~\cite{darmon_pollack} that there is a point
$$
Q = \Bigl( 2\,\xi^6 - 4\, \xi^5 -14\, \xi^4 +17\, \xi^3 + 30\, \xi^2 -7\, \xi -3,\ 2\,\xi^6-19\,\xi^4-2\,\xi^3+32\,\xi^2+\xi-5\Bigr)
$$
of infinite order on $A$ defined over $L=\QQ(\xi)$ obtained from Darmon's construction of modular points. The trace of $Q$ in $L/\QQ$ is equal to the generator $R =(0,0)\in A(\QQ)$. Let $\sigma$ be a generator of $P$. It is easy to check that
$$
 \ZZ\, R \oplus \ZZ \, Q \oplus \ZZ\,(\sigma + \sigma^{-1})Q \oplus \ZZ\, (\sigma^2 + \sigma^{-2}) Q
$$
 has finite index coprime to $p$ in $A(L)$. Hence we can take $Q$ as the point whose existence is predicted by Proposition~\ref{Q_prop}. (In the general case, we may have to take a linear combination of a new point in $A(L)$ and the generator in $A(\QQ)$ to assure that the trace generates $A(\QQ)$.) In fact, in our case, we can check that $A(F) = \ZZ[G] Q$ by using the bound given in~\cite{cps}.

Using modular symbols for the character $\epsilon$ and a Heegner point computation for $\eins$, we can prove that the formulae~\eqref{Qhat_eq} evaluate to

\begin{equation*}
  \hat{\mathcal{Q}}_{\eins} = \frac{L'(A,1)}{\Omega^+(A)\cdot H_{\eins}} = 1
     \qquad\text{ and }\qquad
  \hat{\mathcal{Q}}_{\epsilon} = \frac{L(A,\epsilon,1)\cdot \sqrt{d_K}}{\Omega^+(A)} = 4
\end{equation*}
and hence conclude that BSD holds for $A_{/\QQ}$ and $A_{/K}$ with $\sha(A_{\QQ}) = \sha(A_{K}) = 0$. Next, we compute a numerical approximation to $L'(A,\psi,1)$ for a $2$-dimensional representation $\psi\in\Ir(G)$.
The corresponding value of $\hat{\mathcal{Q}}_{\psi}$ is equal to $4.0000000000000$ for all such $\psi$, but we know of no means of proving that this value is indeed algebraic and equal to the value $4$. It predicts with good accuracy that BSD for $A_{/L}$ would imply that $\sha(A_{L})$ and $\sha(A_{F})$ are trivial. However assuming that $\hat{\mathcal{Q}}_{\psi} = 4$, we compute  the $S_{\rm r}$-truncated version
\begin{equation*}
  \mathcal{Q}_{\eins} = -\frac{578}{577}\  \text{ and }\ \mathcal{Q}_\epsilon = 4\ \text{ and } \ \mathcal{Q}_{\psi} = -\frac{2312}{577}\ \text{ for all $\dim(\psi) = 2$}
\end{equation*}
and hence find that the congruence~\eqref{prime_cong_eq} holds modulo $p=7$. Note that in this particular case, the values $\hat{\mathcal{Q}}_\psi$ were all in $\QQ$. In other words we have found convincing numerical evidence that a Gross-Zagier formula

\begin{equation*}
  \frac{L'(A_{/K},\chi,1) \sqrt{d_K}}{\Omega^+(A)^2} = \frac{c_{\infty}^2}{c^2} \cdot \frac{1}{\vert P \vert} \bigl\langle T_{\chi}(Q), T_{\check\chi}(Q)\bigr\rangle_F
\end{equation*}
analogous to~\eqref{GZ form} should hold for all $\chi\in\Ir(P)$ because $c_{\infty} = 2$ and $c=1$.

\subsubsection{A quintic example}
As a second example, we consider the curve
$$
A\colon\qquad y^2\ +\ x\, y\ =\ x^3\ -\ 4\,x\ -\ 1
$$
labelled 21a1 in~\cite{cremona}. It has rank $0$ over $\QQ$, but rank $1$ over $K=\QQ(i)$ and the group $A(K)$ is generated by the point
$$
R = \bigl( \tfrac{3}{2},\ \tfrac{-3+7\,i}{4}  \bigr).
$$
Now we consider the extension $F/K$ given by a solution $\xi$ of the polynomial
$$
 x^5 - 2\, x^4 - 6\, x^3 +10\, x^2 +17\, x - 12
$$
The extension $F/K$ is only ramified at the place $19\, \ZZ[i]$. All of our hypotheses except~\ref{hyp_g} can be verified to hold in this example.

By a simple search for points, we find the point
$$
T = \Bigl( \tfrac{1}{8} \bigl( -\xi^4 -\xi^2 -12\,\xi-8\bigr),\ \tfrac{1}{32}\big( 39\xi^4 - 7\,\xi^3 -213\xi^2 -127\,\xi+196\bigr)\Bigr)
$$
of infinite order defined over $L=\QQ(\xi)$. The bounds in~\cite{cps} can then be used to prove that
$ \ZZ \, T \oplus \ZZ\, (\sigma + \sigma^{-1}) T$ generates $A(L)_{\rm tf}$ where $\sigma$ is a generator of $P$.
To find the point $Q$ is a bit more elaborate than in the previous example as $\rho_A = \eins$. In fact the $\ZZ[G]$-module generated by $T$ and $R$ will have index
$p$ in $A(F)$ because \cite[Corollary 2.5]{Selmerstr} tells us that $A(F)_p$ is projective. We are going to use the relation
$$
 p\cdot \sigma^{(p+1)/2} = \Tr_{F/K} \ + \ (\sigma - 1) \sum_{i=1}^{(p-1)/2} i\, \bigl(\sigma^i - \sigma^{-i}\bigr)
$$
in $\ZZ[G]$ which is reminiscent of Kolyvagin's derivative construction. We now try to find a point $T' = a\,T+b\, (\sigma+\sigma^{-1})T$ in $A(L)$ with $0\leq a,b < p$
such that $O\neq Q' = R +(\sigma - 1)(T')$ is divisible by $p$ in $A(F)$. Then we can take $Q$ such that $p\sigma^{(p+1)/2}(Q) = Q'$ to be the point predicted by
Proposition~\ref{Q_prop}. In our concrete case this works with $(a,b) = (4,3)$. It can be shown that
$A(F) = {}^{\ZZ}\!/\!{}_{2\ZZ} \oplus {}^{\ZZ}\!/\!{}_{4\ZZ} \oplus \ZZ[G] Q$ as a $\ZZ[G]$-module.

Using modular symbols and Heegner points, we can provably compute that
\begin{equation*}
 \hat{\mathcal{Q}}_{\eins} = \frac{1}{4} \qquad\text{ and }\qquad \hat{\mathcal{Q}}_{\epsilon} = \frac{1}{2}
\end{equation*}
and hence deduce that BSD holds for $A_{/\QQ}$ and $A_{/K}$ with trivial Tate-Shafarevich groups in both cases.
We compute to a high precision the derivatives of $L(A,\psi,s)$ for the representations $\psi=\psi_1$ and $\psi_2$ of dimension two and we find, with an error less than
$10^{-28}$, that $\hat{\mathcal{Q}}_{\psi_1}\cdot \hat{\mathcal{Q}}_{\psi_2} \approx 256$, predicting that $|\sha(A_L)| = 4$ and
$\hat{\mathcal{Q}}_{\psi_1} + \hat{\mathcal{Q}}_{\psi_2} \approx 48$.
It also predicts that $\sha(A_{F})$ has order $32$. We will now assume that these are actually equalities and conclude that
$\hat{\mathcal{Q}}_{\psi_i} = 8\cdot \bigl(3\pm\sqrt{5}\bigr)$. One then computes the $S_{\rm r}$-truncated values to be
\begin{equation*}
  \mathcal{Q}_{\eins} = \frac{8}{19},\qquad \mathcal{Q}_{\epsilon} = \frac{24}{19} \qquad \text{ and }\qquad \mathcal{Q}_{\psi_1} + \mathcal{Q}_{\psi_2} = -96
\end{equation*}
and this shows that the congruence~\eqref{prime_cong_eq} holds modulo $p=5$.


\bibliographystyle{amsplain}
\bibliography{mwc}

\end{document}